\documentclass{amsart}
\usepackage{mathrsfs}
\usepackage{epic,eepic,epsf,epsfig,tikz}
\usetikzlibrary{positioning}
\usepackage{amsfonts,srcltx,mathrsfs}
\usepackage{chngcntr}
\counterwithin{table}{section}
\usepackage[margin=1.32cm]{geometry}
\usepackage{multirow}
\usepackage{amsmath}
\usepackage{amssymb}
\usepackage{amsbsy}
\usepackage{graphicx}
\usepackage{amsfonts}
\usepackage{color}
\usepackage{makecell}
\newtheorem{lem}{Lemma}[section]%
\newtheorem{theorem}[lem]{Theorem}%
\counterwithin{figure}{section}
\newtheorem{prop}[lem]{Proposition}%
\newtheorem{rem}[lem]{Remark}%

\newtheorem{notation}[lem]{Notation}

  \def\G{\Gamma}

\def\nd{\mathrel{\bigm|\kern-.7em/}}

\def\f{\noindent}

\def\P\GammaL{\hbox{\rm P\Gamma L}}

\def\Aut{\hbox{\rm Aut}}

\def\Cay{\hbox{\rm Cay}}
\def\BiCay{\hbox{\rm BiCay}}

\def\mz{{\mathbb Z}}

\begin{document}
\title[]{On oriented $m$-semiregular representations of finite groups}

\author{Jia-Li Du}
\address{Jia-Li Du, School of Mathematics, China University of Mining and Technology, Xuzhou 221116, China}
\email{dujl@cumt.edu.cn}

\author{Yan-Quan Feng}
\address{Yan-Quan Feng, School of mathematics and statistics, Beijing Jiaotong University, Beijing 100044, China}
\email{yqfeng@bjtu.edu.cn}

\author{Sejeong Bang}
\address{Sejeong Bang, Department of Mathematics, Yeungnam University, Gyeongsan 38541, South Korea}
\email{sjbang@ynu.ac.kr}

\date{}
 \maketitle

\begin{abstract}

A finite group $G$ admits an {\em oriented regular representation} if there exists a Cayley digraph of $G$ such that it has no digons and its automorphism group is isomorphic to $G$. Let $m$ be a positive integer. In this paper, we extend the notion of oriented regular representations to oriented $m$-semiregular representations using $m$-Cayley digraphs. Given a finite group $G$, an {\em $m$-Cayley digraph} of $G$ is a digraph that has a group of automorphisms isomorphic to $G$ acting semiregularly on the vertex set with $m$ orbits. We say that a finite group $G$ admits an {\em oriented $m$-semiregular representation} if there exists a regular $m$-Cayley digraph of $G$ such that it has no digons and $G$ is isomorphic to its automorphism group. In this paper, we classify finite groups admitting an oriented $m$-semiregular representation for each positive integer $m$.\\

\f {\bf Keywords:} Regular group, semiregular group, regular representation, $m$-Cayley digraph, ORR, O$m$SR.
\medskip

\f {\bf 2010 Mathematics Subject Classification:} 05C25, 20B25.
\end{abstract}

\section{Introduction}
All groups and digraphs considered in this paper are finite. For a digraph $\G$, we denote by $V(\G)$ and $A(\G)$ the set of vertices and the set of arcs respectively, where $V(\G)\ne \emptyset$ and $A(\G)\subseteq V(\G) \times V(\G)$. A digraph $\Gamma $ is called a {\em graph} if the binary relation $A(\G)$ is symmetric, that is, $A(\G)=\{(y,x)\ |\ (x,y)\in A(\G) \}$. A digraph is called {\em regular} if each vertex has the same in- and out-valency. Let $\G$ be a digraph. An automorphism of $\Gamma$ is a permutation $\sigma$ of $V(\G)$ fixing $A(\G)$ setwise, that is, $(x^\sigma , y^\sigma ) \in A(\G)$ if and only if $(x, y) \in A(\G)$. The full automorphism group of $\G$ is denoted by $\Aut(\G)$.\\

\indent Let $G$ be a group with a subset $R\subseteq G\setminus \{1\}$. A {\em Cayley digraph} $\G=\Cay(G,R)$ is the digraph with $V(\Gamma)=G$ and
$A(\G)=\{(g,rg) ~|\ g\in G,r\in R\}$. In particular, $\G$ is a Cayley graph if
and only if $R=R^{-1}$. The right regular representation of $G$ gives rise to an
embedding of $G$ into $\Aut(\Gamma)$, and thus we identify $G$ with its image under this
permutation representation. We say that a group $G$ admits a {\em (di)graphical
regular representation} (GRR or DRR for short) if there exists a Cayley (di)graph
$\G$ of $G$ satisfying $G \cong \Aut(\G)$. Babai~\cite{Babai} proved that except $Q_8$, $\mz_2^2$, $\mz_2^3$, $\mz_2^4$ and $\mz_3^2$, each finite group admits a DRR. It is clear that if a group $G$ admits a GRR then $G$ admits a DRR. However, the converse is not true. Many researchers proposed and investigated various kinds of generalizations of the classifications of DRRs and GRRs. Despite the classification of groups admitting a DRR by Babai~\cite{Babai}, the classification of groups admitting a GRR needs considerably more work. For more results along this way, we refer the reader to~\cite{Godsil,Imrich,ImrichWatkins,NowitzWatkins1,NowitzWatkins2}. A \emph{tournament} is a digraph $\G$ such that for any two distinct vertices $x, y \in  V(\G)$, exactly one of $(x, y)$ and $(y, x)$ is in $A(\G)$. Note that Cayley digraph $\Cay(G, R)$ is a tournament if and only if $R\cap R^{-1} =\emptyset$ and $R\cup R^{-1}=G \setminus\{1\}$, where $1$ is the identity of $G$. It is known that if $\Cay(G, R)$ is a tournament then the order of $G$ is odd. We say that a group $G$ admits a {\em tournament regular representation} (TRR for short) if there exists a Cayley digraph $\G$ of $G$ such that $\G$ is a tournament satisfying $G\cong \Aut(\G)$. Babai and Imrich~\cite{BabaiI} proved that
except $\mz_3^2$, each group of odd order admits a TRR. An \emph{oriented Cayley digraph} is a Cayley digraph $\Cay(G, R)$ satisfying $R \cap R^{-1} =\emptyset$ (i.e., it has no digons.). We say that a group $G$ admits an {\em oriented regular representation} (ORR for short) if there exists a Cayley digraph $\G=\Cay(G,R)$ of $G$ such that $R\cap R^{-1}=\emptyset$ and $G\cong \Aut(\G)$ both hold. It is clear that if a group $G$ admits a TRR then it admits an ORR. In 2017, Morris and Spiga~\cite{MorrisSpiga1} proved that each finite non-solvable group admits an ORR. Spiga~\cite{Spiga} proved that generalized dihedral groups do not admit ORRs. Morris and Spiga~\cite{MorrisSpiga3} completely classified the finite groups admitting ORRs. For more results, we refer the reader to \cite{Hujdurovic,KMMS,Leemann,MorrisSpiga4,Spiga2,VerretXia,Xia}.\\

\indent Let $G$ be a permutation group on a set $\Omega$ and let $\omega \in \Omega$. Denote by $G_\omega$ the stabilizer of $\omega$ in $G$ (i.e., the subgroup
of $G$ fixing $\omega$.). We say that $G$ is {\em semiregular} on $\Omega$ if $G_\omega = 1$ for every $\omega \in \Omega$, and {\em regular} if it is semiregular and transitive. For a positive integer $m$, an {\em $m$-Cayley (di)graph} of a group $G$ is defined as a (di)graph which has a semiregular group of automorphisms
isomorphic to $G$ with $m$ orbits on its vertex set. Note that $1$-Cayley (di)graphs are the usual Cayley (di)graphs. For $m=2$, $2$-Cayley (di)graphs
are called {\em bi-Cayley (di)graphs} (For more concrete definition, see Section 2.). We say that a group $G$ admits a {\em (di)graphical $m$-semiregular representation} (G$m$SR or D$m$SR for short) if there exists a regular $m$-Cayley (di)graph $\Gamma$ of $G$ satisfying $G\cong \Aut(\Gamma)$. In particular, a G$1$SR (D$1$SR, respectively) is the usual GRR (DRR, respectively). In~\cite{DFS}, Du, Feng and Spiga classified finite groups admitting a G$m$SR or a D$m$SR for each positive integer $m$.\\

\indent Now we consider $n$-partite digraphs. We say that a group $G$ admits an {\em $n$-partite digraphical representation} ($n$-$\rm PDR$ for short) if there exists a regular $n$-partite digraph $\Gamma$ such that $G\cong \mathrm{Aut}(\Gamma)$, $\mathrm{Aut}(\Gamma)$ acts semiregularly on $V(\G)$ and the orbits of $\mathrm{Aut}(\Gamma)$ on $V(\G)$ form a partition into $n$ parts giving the structure of $\Gamma$. It is known that each $n$-PDR is also a D$n$SR, but the converse is not true. Recently, Du, Feng and Spiga~\cite{DFS2,DFS3} classified the finite groups admitting an $n$-PDR for each positive integer $n$. \\

\indent In this paper, we extend the notion of oriented regular representations to oriented $m$-semiregular representations using $m$-Cayley digraphs. We say that for a positive integer $m$, a group $G$ admits an {\em oriented $m$-semigular representation} (O$m$SR for short) if there exists a regular oriented $m$-Cayley digraph $\G$ of $G$ satisfying $G\cong \Aut(\G)$. In particular, an O$1$SR is the usual ORR. Finite groups admitting ORRs are completely classified (see Proposition \ref{prop=ORR},  \cite{MorrisSpiga3}).\\

\indent Now we introduce the main result of this paper. In Theorem \ref{theo=main}, we classify the finite groups admitting an oriented $m$-semiregular representation $(m\geq 1)$ as follows.

\begin{theorem}\label{theo=main}
Let $m$ be a positive integer and let $G$ be a finite group. Then $G$ satisfies one of the following:
\begin{enumerate}
\item $G$ admits an O$m$SR;
\item $m=1$, and $G$ is isomorphic to either a generalized dihedral group of order greater than $2$ or one of the $11$ exceptional groups listed in {\sc Table}~$\ref{table1}$;
\item $m=2$, and $G$ is isomorphic to either $\mz_1$, $\mz_2$, $\mz_2^2$,
$\mz_2^3$ or $\mz_2^4$;
\item $3\leq m\leq 6$, and $G$ is isomorphic to $\mz_1$.
\end{enumerate}
\end{theorem}

This paper is organized as follows. In Section 2, we set up notations and review known results which will be used in this paper. In Section 3, we prove Theorem \ref{theo=main}. We first show in Lemma \ref{lem=ORR} that any finite group $G$ of order at least $3$ admitting an ORR  must admit an O$m$SR for any integer $m\geq 2$. To consider finite groups which do not admit ORRs, we mainly consider elementary abelian $2$-groups and non-abelian generalized dihedral groups. First, we consider when elementary abelian $2$-groups admit an O$m$SR for positive integer $m$. To show this, we divide into two cases: groups of order at most $16$ (Lemma~\ref{lem=elementary abelian 2-group order le 16}); groups of order greater than $16$  (Lemma~\ref{lem=elementary abelian 2-group}). On the other hand, we prove in Lemma \ref{lem=generalizeddihedral} that finite non-abelian generalized dihedral groups always admit an O$m$SR for any integer $m\geq 2$. Using these results, we classify the finite groups admitting an oriented $m$-semiregular representation for each positive integer $m$ (Theorem \ref{theo=main}).

\section{ORRs, $m$-Cayley digraphs and notations}
In this section, we set up notations and review several results which will be used in this paper.\\

As usual, for a positive integer $n$ and a prime $p$, we denote by $\mz_n$,  $\mz_p^n$,  $Q_8$   and $D_n$ the cyclic group of order $n$, the elementary abelian group of order $p^n$, the quoternion group of order $8$ and the dihedral group of order $2n$, respectively. For an abelian group $H$, the {\em generalized dihedral group} over $H$ is the semidirect product of $H$ with the cyclic group of order two, with the non-identity element acting as the inverse map on $H$. In the following remark, we note some properties of generalized dihedral groups which will be used later.

\begin{rem}\label{rmkgdg}
Let $G$ be the generalized dihedral group over an abelian group $H$. Then we have the following.
\begin{enumerate}
\item The exponent of $G$ is the least common multiple of $2$ and the exponent of $H$.
\item $G$ is abelian if and only if $G$ is an elementary abelian $2$-group.
\end{enumerate}
\end{rem}

To state finite groups admitting ORRs, we first give the following table which shows $11$ exceptional groups admitting no ORRs (see \cite{MorrisSpiga3}).

\begin{table}[!ht]
\begin{tabular}{|c|c|}\hline
\rule{0pt}{12pt}{Group}&\rule{0pt}{12pt}{Comments}\\\hline

\rule{0pt}{12pt}\makecell*[c]{$Q_8,\ \mz_4 \times \mz_2, \ \mz_4 \times \mz_2^2, \ \mz_4 \times \mz_2^3, \ \mz_4 \times \mz_2^4, \ \mz_3^2, \ \mz_3 \times \mz_2^3$}&\\\hline
\rule{0pt}{12pt}\makecell*[c]{$H_1:=\langle x, y~|~x^4=y^4=(xy)^2=(xy^{-1})^2=1\rangle$}& order $16$\\\hline
\rule{0pt}{12pt}\makecell*[c]{$H_2:=\langle x,y,z~|~x^4 =y^4=z^4=(yx)^2=(yx^{-1})^2=(yz)^2=(yz^{-1})^2=1,$$\,\,\,
  $$x^2=y^2=z^2,x^z=x^{-1} \rangle$}& order $16$\\\hline

\rule{0pt}{12pt}\makecell*[c]{$H_3:=\langle x,y,z~|~x^4 = y^4 = z^4 = (xy)^2 = (xy^{-1})^2 =(xz)^2= (xz^{-1})^2 =(yz)^2 = (yz^{-1})^2 = x^2y^2z^2 = 1\rangle$ }& order $32$\\\hline
\rule{0pt}{12pt}\makecell*[c]{ $D_4\circ D_4$ (the central product of $D_4$ and $D_4$)}& order 32 \\\hline
\end{tabular}
\vskip0.2cm
\caption{The exceptional groups admitting no ORRs}\label{table1}
\end{table}

The following is the classification of finite groups admitting ORRs.

\begin{prop}\label{prop=ORR}{\rm \cite[Theorem 1.2]{MorrisSpiga3}}
Let $G$ be a finite group. Then $G$ admits an ORR, except
\begin{itemize}
  \item[\rm(1)] $G$ is a generalized dihedral group of order greater than $2$;
  \item[\rm(2)] $G$ is isomorphic to one of the 11 exceptional groups given in {\sc Table}~$\ref{table1}$.
\end{itemize}
\end{prop}

\indent Let $m\geq 1$ be an integer. Recall that an {\em $m$-Cayley digraph} of a finite group $G$ is defined as a digraph which has a semiregular group of automorphisms isomorphic to $G$ with $m$ orbits on its vertex set. Now, we consider more concrete definition for $m$-Cayley digraphs. For each $i \in \mz_m$, put $G_i:=\{g_i\ |\ g\in G\}$.  For any subset $H\subseteq G$ and $i\in \mz_m$, we also denote $H_i:=\{h_i\ |\ h\in H\}$. Similar to Cayley digraphs, an  $m$-Cayley digraph $\Gamma$ can be viewed as the digraph $$\Cay(G,T_{i,j}: i,j\in\mz_m)$$
with the vertex set and the arc set:
\[
\bigcup_{i\in \mz_m}  G_i \mbox{   and  }
\bigcup_{i,j\in\mz_m} \{(g_i, (tg)_j)~|~t\in T_{i,j}\}\]
respectively, where $T_{i,j}~(i,j\in \mz_m)$ are subsets of $G$ satisfying $1\notin T_{i,i}$. In particular, the above digraph $\Cay(G,T_{i,j}: i,j\in\mz_m)$ is also an $m$-Cayley digraph of a group $G$ (with respect to $T_{i,j} ~(i,j\in\mz_m)$ ). Note here that an $m$-Cayley digraph $\Cay(G,T_{i,j}: i,j\in\mz_m)$ is an oriented $m$-Cayley digraph if and only if $T_{i,j}\cap T_{j,i}^{-1}=\emptyset$ for all $i,j\in \mz_m$. Without loss of generality, let $\Gamma=\Cay(G,T_{i,j}: i,j\in\mz_m)$ be an $m$-Cayley (di)graph of a group $G$ with respect to $T_{i,j} ~(i,j\in\mz_m)$. For any element $g\in G$, the right multiplication $R(g)$ (i.e., mapping each vertex $x_i\in G_i$ to $(xg)_i\in G_i$ for all $i\in\mz_m$) is an automorphism of $\G$, and $R(G)$ is a semiregular group of automorphisms of $G$ with $G_i$ as orbits, where
$$R(G):=\{ R(g)\ |\ g\in G\}.$$
In particular, $2$-Cayley (di)graph $\Cay(G,T_{i,j}: i,j\in\mz_2)$ is called a {\em bi-Cayley (di)graph} and it is also denoted by $$\BiCay(G,R,L,S,T)$$
where $R=T_{0,0}$, $L=T_{1,1}$, $S=T_{0,1}$ and $T=T_{1,0}$.\\

The following proposition gives a result of the automorphisms between bi-Cayley digraph $\BiCay(G,\emptyset,\emptyset,S,T)$ and Cayley digraphs $\Cay(G,S)$ and $\Cay(G,T)$.

\begin{prop}\label{prop=Haar}{\rm \cite[Lemma 3.2]{DFS2}}
Let $G$ be a finite group and let $\sigma $ be a permutation of $G$.
Denote by $\sigma'$ the permutation of $G_0\cup G_1$ mapping
$g_i$ to $g^{\sigma}_i$ for each $g\in G$ and $i\in \mz_2$.
Then the following are equivalent.
\begin{enumerate}
\item $\sigma\in \Aut(\Cay(G,S))\cap\Aut(\Cay(G,T))$
\item $\sigma'\in\Aut(\BiCay(G,\emptyset,\emptyset,S,T))$
\end{enumerate}
\end{prop}

To end this section, we set up some notations for this paper.

\begin{notation}\label{arcdef}
\begin{enumerate}
\item Let $\G$ be a (di)graph. For any subset $X$ of the vertex set $V(\Gamma)$, denote by $\Gamma[X]$ the induced sub-(di)graph by $X$ in $\Gamma$. We simply write $[X]$ for $\Gamma[X]$ if there are no confusions.
\item Let $\G$ be a digraph. For a vertex $x\in V(\G)$, denote by $\G^+(x)$ and $\G^-(x)$ the out- and in-neighbors of $x$ in $\G$, respectively. For a subset $X\subseteq V(\G)$, define
$$\G ^{+}(X):=\bigcup_{x\in X}\G^{+}(x) \mbox{~~~and~~~}\G ^{-}(X):=\bigcup_{x\in X}\G^{-}(x).$$
\item Let $\G$ be a digraph and let $X,Y$ be subsets of $V(\G)$. Recall that $A([X])$  is the set of arcs of the induced sub-digraph $[X]$ by $X$ in $\G$ (i.e., the set of arcs of $\G$ in $X\times X$). Denote by $A(X,Y)$ the set of arcs of $\G$ in $X\times Y$. If $X=Y$ then we simply write $A(X)$ instead of $A(X,X)$. If $X=\{x\}$ or $Y=\{y\}$ then we simply write
$$A(x,Y):=A(\{x\},Y), ~~A(X,y):=A(X,\{y\}) \mbox{~~and~~} A(x,y):=A(\{x\},\{y\})$$
 if there are no confusions.
\end{enumerate}
\end{notation}

\section{Proof of Theorem~\ref{theo=main}}

In this section, we prove Theorem~\ref{theo=main}. We first consider the case when given group admits an ORR.

\begin{lem}\label{lem=ORR}
Let $G$ be a finite group of order at least $3$. If $G$ admits an ORR, then $G$ admits an O$m$SR for any integer $m\geq 2$.
\end{lem}

\begin{proof}
Assume that $G$ admits an ORR. That is, there exists a subset $R\subseteq G$ such that $\Cay(G,R)$ is an ORR of $G$. If $R=\emptyset$ then $\Cay(G,R)$ is the union of $|G|$ isolated vertices. Since $\Cay(G,R)$ is an ORR, $G$ is $\mz_1$ or $\mz_2$ which is impossible as $|G|\geq 3$. Thus $$R\neq\emptyset.$$
Let $\Sigma:=\Cay(G,R)$, $\Phi:=\Cay(G,R^{-1})$ and $E_H:=\{(x,y,z)\ |\ x,y,z\in H, xy=z\}$ for $H\in \{R, R^{-1}\}.$  As $(x,y,z)\in E_R$ if and only if $(y^{-1},x^{-1},z^{-1})\in E_{R^{-1}}$, we have $|E_R|=|E_{R^{-1}}|$. Put $k:=|E_R|=|E_{R^{-1}}|$.
As $|A([\Sigma^+(1)])|=|E_R|$ and $|A([\Phi^+(1)])|=|E_{R^{-1}}|$ hold (see Notation \ref{arcdef}), we obtain
\begin{equation}\label{ak}
|A([\Sigma^+(1)])|=|A([\Phi^+(1)])|=k
\end{equation}
where $1$ is the identity of $G$. As $R\neq\emptyset$, take $a\in R$. Note that
\begin{eqnarray}\label{Eq1}
R\cap R^{-1}=\emptyset, \ \ a\in R, \ \ a^{-1}\not\in R,\ \ o(a)\geq 3,\ \ 1\not\in R\cup R^{-1}
\end{eqnarray}
all hold, where $o(a)$ is the order of $a$ in $G$. Now, we divide the proof into two cases: $m=2$ (Case 1) and $m\geq 3$ (Case 2). \\

\f{\bf Case 1:} Show that $G$ admits an O$2$SR. \\
\noindent{\em Proof of Case 1. }Let $\G_2:=\BiCay(G,R,R^{-1},\{1\},\{a^{-1}\})$ and $\mathcal{A}:=\Aut(
\G_2)$. Then $\G_2$ is an oriented $2$-Cayley digraph of $G$ with valency $|R|+1$ satisfying
$\G_2^+(1_0)=\{r_0,1_1~|~r\in R\}, \G_2^+(1_1)=\{r^{-1}_1,a^{-1}_0~|~r\in R\}$ and $\G_2^{-}(1_1)=\{r_1,1_0~|~r\in R\}.$\\
We first show that $\mathcal{A}$ is not vertex-transitive by considering $|A([\G_2^+(1_0)])|$ and $|A([\G_2^+(1_1)])|$. As $[\G_2^+(1_0)\cap G_0]\cong [\Sigma^+(1)]$, we need to count the number of arcs between $r_0$ and $1_1$ for $r\in R$. Since
$$\G_2^+(1_1)\cap \G_2^+(1_0)=\emptyset \ \mbox{ and } \ \G_2^{-}(1_1)\cap \G_2^+(1_0)=\emptyset$$ holds (see (\ref{Eq1})), $1_1$ is an isolated vertex in $[\G_2^+(1_0)]$. Hence $|A([\G_2^+(1_0)])|=|A([\Sigma^+(1)])|=k$ follows by (\ref{ak}). On the other hand, $|A([\G_2^+(1_1)])|\geq |A([\Phi^+(1)])|+1=k+1$ holds since $(a^{-1}_0,a^{-1}_1)$ is an arc in $[\G_2^+(1_1)]$ by $T_{0,1}=\{1\}$. Therefore $[\G_2^+(1_0)]\not\cong[\G_2^+(1_1)]$ and $\mathcal{A}$ is not vertex-transitive. Recall that $R(G)$ is a group of automorphisms of $\G_2$ with two orbits $G_0$ and $G_1$, and $\mathcal{A}$ also has two orbits $G_0$ and $G_1$ satisfying $\mathcal{A}=R(G)\mathcal{A}_{1_0}$. Since $\mathcal{A}_{1_0}$ fixes $G_0$ setwise and $[G_0]\cong \Sigma $ is an ORR, $\mathcal{A}_{1_0}$ fixes $G_0$ pointwise. Moreover, $\mathcal{A}_{1_0}$ fixes $G_1$ pointwise as $T_{0,1}=\{1\}$. Therefore, $\mathcal{A}_{1_0}=1$ and $\mathcal{A}=R(G)\mathcal{A}_{1_0}=R(G)$. This shows that $\G_2$ is an O$2$SR of $G$. This completes the proof of Case 1.\qed \\

\f{\bf Case 2:} For $m\ge3$, show that $G$ admits an O$m$SR.\\
\noindent{\em Proof of Case 2. } For $m\ge3$, define a subset $T_{i,j}\subseteq G~(i,j\in \mz_m)$ as follows:
\[
\begin{array}{ll}
T_{0,1}=\{1\},~\,\,\,\,\,\,\,\,\,\,T_{1,0}=\{a^{-1}\},~T_{1,1}=R^{-1},~T_{m-1,0}=\{a\}~ ; &\nonumber \\
T_{i,i-1}=\{1\},~\,\,\,\,\,T_{i,i}=R  &\mbox{ for } i\neq 1~; \nonumber\\
T_{i,i+1}= \{a^{-1}\} &\mbox{ for } i\ne 0,m-1~; \nonumber \\
T_{i,j}=\emptyset  &\mbox{ for } i\ne j, j\pm1. \nonumber
\end{array}
\]
Let $\G_m:=\Cay(G,T_{i,j}:i,j \in \mz_m)$. Then $\G_m$ is an oriented $m$-Cayley digraph of $G$ with valency $|R|+2$. To complete the proof for Case 2, we need the following claim that counts the number of arcs in the induced digraph $[\G_m^+(1_i)]$ for $i\in \mz_m$.

\smallskip
\f{\bf Claim:}
\smallskip

\f(i) $|A([\G_3^+(1_{0})])|=k+2$, $|A([\G_m^+(1_{0})])|=k+1$ $(m\geq4).$\\
(ii) $|A([\G_3^+(1_{1})])|=k+3$, $|A([\G_m^+(1_{1})])|=k+2$ $(m\geq4).$\\
(iii) $|A([\G_m^+(1_{m-1})])|=k+1$ $(m\geq 3).$\\
(iv) $|A([\G_m^+(1_{i})])|=k~(2\leq i\leq m-2 ).$
\medskip

\noindent{\em Proof of Claim}.
\begin{figure}
  \centering
  \includegraphics[width=10cm]{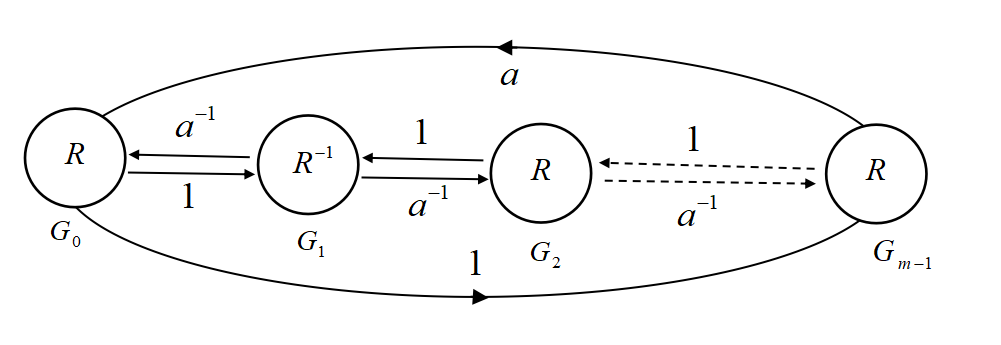}
\caption{The digraph $\G_m$~for $m\geq 3$}\label{Fig1}
\end{figure}
We first need to consider the set of out- and in- neighbors of $1_i$, $a_i$ and $a_i^{-1}$ in $\Gamma_m$ for $i\in \mz_m$. With the aid of {\sc Figure}~\ref{Fig1}, we find the following:
\[\begin{array}{ll}
\G_m^+(1_0)=\{r_{0},1_1,1_{m-1}~|~r\in R\},&~~\G_m^-(1_0)=\{r^{-1}_{0},a_1,a^{-1}_{m-1}~|~r\in R\},\\
\G_m^+(1_1)=\{r^{-1}_{1},a^{-1}_2,a^{-1}_{0}~|~r\in R\},&~~\G_m^-(1_1)=\{r_1,1_2,1_0~|~r\in R\}, \\
\G_m^+(1_{m-1})=\{r_{m-1},a_0,1_{m-2}~|~r\in R\},&~~ \\
\G_m^+(1_i)=\{r_{i},a^{-1}_{i+1},1_{i-1}~|~r\in R\} ~(2\leq i\leq m-2 ),&~~\G_m^-(1_i)=\{r^{-1}_{i},1_{i+1},a_{i-1}~|~r\in R\} ~(i\not=0,1),\\
\G_m^+(a^{-1}_0)=\{(ra^{-1})_{0},a^{-1}_1,a^{-1}_{m-1}~|~r\in R\},& ~~\G_m^-(a^{-1}_0)=\{(r^{-1}a^{-1})_{0},1_1,a^{-2}_{m-1}~|~r\in R\},\\
\G_m^+(a_{m-1}^{-1})=\{(ra^{-1})_{m-1},1_0,a^{-1}_{m-2}~|~r\in R\},& ~~\\
\G_m^+(a^{-1}_{i+1})=\{(ra^{-1})_{i+1},a^{-2}_{i+2},a^{-1}_{i}~|~r\in R\} ~(1\leq i\leq m-3),&~\G_m^-(a_{i+1}^{-1})=\{(r^{-1}a^{-1})_{i+1},a^{-1}_{i+2},1_{i}~|~r\in R\}~(i\ne 0,m-1),\\
\G_m^+(a_0)=\{(ra)_{0},a_1,a_{m-1}~|~r\in R\},& ~~\G_m^-(a_0)=\{(r^{-1}a)_{0},a^2_1,1_{m-1}~|~r\in R\}.\\
\end{array}
\]

The above sets enable us to prove the claim as follows.

\noindent{\em Proof of (i). }It is easy to find
\[\begin{array}{llll}
\G_3^+(1_1)\cap \G_3^+(1_0)=\emptyset,& \G_3^+(1_2)\cap \G_3^{+}(1_{0})=\{a_0,1_1\},&\G_3^-(1_1)\cap \G_3^+(1_0)=\{1_2\}, &\G_3^-(1_2)\cap \G_3^{+}(1_{0})=\emptyset~;\\
\G_m^+(1_1)\cap \G_m^+(1_0)=\emptyset,&\G_m^+(1_{m-1})\cap \G_m^{+}(1_0)=\{a_0\},&
\G_m^-(1_1)\cap \G_m^+(1_0)=\emptyset,&\G_m^-(1_{m-1})\cap \G_m^{+}(1_0)=\emptyset~(m\geq 4).
\end{array}
\]
As $|A([\G_m^+(1_0)])|=|A(R_0)|+|A(1_1,R_0\cup \{1_{m-1}\})|+|A(1_{m-1},R_0\cup\{1_1\})|-|A(1_1,1_{m-1})|$ holds (see Notation \ref{arcdef}), the result (i) now follows by
\begin{eqnarray*}
|A(R_0)|&=&|[\G_m^+(1_0)\cap G_0]|= |[\Sigma^+(1)]|=k~;\\
|A(1_1,R_0\cup \{1_{m-1}\})|&=&|\G_m^+(1_1)\cap \G_m^+(1_0)|+|\G_m^-(1_1)\cap \G_m^+(1_0)|~;\\
|A(1_{m-1},R_0\cup\{1_0\})|&=&|\G_m^+(1_{m-1})\cap \G_m^+(1_0)|+|\G_m^-(1_{m-1})\cap \G_m^+(1_0)|~;\\
|A(1_1,1_2)|&=&1~;\\
|A(1_1,1_{m-1})|&=&0
\end{eqnarray*}
for $m\geq 4$.\\

\noindent{\em Proof of (ii). }It is easy to find
\[
\begin{array}{llll}
\G_3^+(a^{-1}_0)\cap \G_3^+(1_1)=\{a_1^{-1},a_2^{-1}\},~ &\G_3^-(a^{-1}_0)\cap \G_3^{+}(1_1)=\emptyset,~&\G_3^+(a^{-1}_2)\cap \G_3^+(1_1)=\{a_1^{-1}\},~&\G_3^-(a^{-1}_2)\cap \G_3^{+}(1_1)=\{a_0^{-1}\}~;\\
 \G_m^+(a^{-1}_0)\cap \G_m^+(1_1)=\{a_1^{-1}\}, ~&\G_m^-(a^{-1}_0)\cap \G_m^{+}(1_1)=\emptyset, ~
&\G_m^+(a^{-1}_2)\cap \G_m^+(1_1)=\{a_1^{-1}\}, ~&\G_m^-(a^{-1}_2)\cap \G_m^{+}(1_1)=\emptyset ~(m\geq 4).\\
\end{array}
\]
The result (ii) follows by
\begin{eqnarray*}
|A([\G_m^+(1_1)])| &= &k-|A(a_0^{-1},a_2^{-1})|+\\
&& |\G_m^+(a_0^{-1})\cap \G_m^+(1_1)|+|\G_m^-(a_0^{-1})\cap \G_m^+(1_1)|+|\G_m^+(a_2^{-1})\cap  \G_m^+(1_1)|+|\G_m^-(a_2^{-1})\cap \G_m^+(1_1)|
\end{eqnarray*}
and $|A(a_0^{-1},a_2^{-1})|=1$ for $m=3$, and $|A(a_0^{-1},a_2^{-1})|=0$ for $m\geq 4$. \\

\noindent{\em Proof of (iii). }For $m\geq 3$, it is easy to find
\begin{eqnarray*}
\G_m^+(a_0)\cap \G_m^+(1_{m-1})=\{a_{m-1}\}, ~\G_m^-(a_0)\cap \G_m^+(1_{m-1})=\emptyset, ~\G_m^+(1_{m-2})\cap \G_m^+(1_{m-1})=\emptyset, ~\G_m^-(1_{m-2})\cap \G_m^+(1_{m-1})=\emptyset.
\end{eqnarray*}
As $|A(a_0,1_{m-2})|=0$ holds by {\sc Figure}~\ref{Fig1}, we obtain (iii):
\begin{eqnarray*}
|A([\G_m^+(1_{m-1})])| &= &k-|A(a_0,1_{m-2})|+|\G_m^+(a_0)\cap \G_m^+(1_{m-1})|+|\G_m^-(a_0)\cap \G_m^+(1_{m-1})|\\&&
+|\G_m^+(1_{m-2})\cap \G_m^+(1_{m-1})|+|\G_m^-(1_{m-2})\cap \G_m^+(1_{m-1})|\\&=&k+1.
\end{eqnarray*}

\noindent{\em Proof of (iv). }Let $m\geq 4$. Since there are no arcs between $1_{i-1}$ and $a^{-1}_{i+1}$ $(2\leq i\leq m-2)$, we obtain
\begin{equation}\label{cliv}
|A(1_{i-1},a^{-1}_{i+1})|=0 ~(2\leq i\leq m-2).
\end{equation}
We divide the proof of (iv) into three cases $i=2$, $i=m-2$ and $3\leq i\leq m-3$.\\
First, we consider $3\leq i\leq m-3$. Then $m\geq 6$ and
\begin{eqnarray*}
\G_m^+(1_{i-1})\cap \G_m^+(1_{i})=\emptyset,~\G_m^-(1_{i-1})\cap \G_m^+(1_{i})=\emptyset,~
\G_m^+(a^{-1}_{i+1})\cap \G_m^+(1_{i})=\emptyset, ~\G_m^-(a^{-1}_{i+1})\cap \G_m^+(1_{i})=\emptyset.
\end{eqnarray*}
Using (\ref{cliv}), we have
\begin{eqnarray*}
|A([\G_m^+(1_i)])| &= &k-|A(1_{i-1},a^{-1}_{i+1})|+|\G_m^+(1_{i-1})\cap \G_m^+(1_{i})|+|\G_m^-(1_{i-1})\cap \G_m^+(1_{i})|
+|\G_m^+(a^{-1}_{i+1})\cap \G_m^+(1_{i})|+|\G_m^-(a^{-1}_{i+1})\cap \G_m^+(1_{i})|\\&=&k
\end{eqnarray*}
and this shows (iv) for $3\leq i\leq m-3$.\\
Now we consider $i=2$. By (\ref{cliv}) and
\begin{eqnarray*}
\G_m^+(1_1)\cap \G_m^+(1_{2})=\emptyset, ~\G_m^-(1_1)\cap \G_m^+(1_2)=\emptyset,~\G_m^+(a^{-1}_{3})\cap \G_m^+(1_{2})=\emptyset,~
\G_m^-(a^{-1}_{3})\cap \G_m^+(1_{2})=\emptyset,
\end{eqnarray*}
we obtain
\begin{eqnarray*}
|A([\G_m^+(1_2)])| &= &k-|A(1_{1},a^{-1}_{3})|+|\G_m^+(1_1)\cap \G_m^+(1_{2})|+|\G_m^-(1_1)\cap \G_m^+(1_2)|+|\G_m^+(a^{-1}_{3})\cap \G_m^+(1_{2})|+|\G_m^-(a^{-1}_{3})\cap \G_m^+(1_{2})|\\&=&k
\end{eqnarray*}
and this shows (iv) for $i=2$.\\
Finally, let $i=m-2$. Note that $m\geq5$ follows by $i=m-2\neq 2$ and $m\geq 4$. By (\ref{cliv}) and
\begin{eqnarray*}
\G_m^+(1_{m-3})\cap \G_m^+(1_{m-2})=\emptyset,~\G_m^-(1_{m-3})\cap \G_m^+(1_{m-2})=\emptyset,~\G_m^+(a^{-1}_{m-1})\cap \G_m^+(1_{m-2})=\emptyset, ~\G_m^-(a^{-1}_{m-1})\cap \G_m^+(1_{m-2})=\emptyset,
\end{eqnarray*}
we obtain (iv):
\begin{eqnarray*}
|A([\G_m^+(1_{m-2})])|&= &k-|A(1_{m-3},a^{-1}_{m-1})|+|\G_m^+(1_{m-3})\cap \G_m^+(1_{m-2})|+|\G_m^-(1_{m-3})\cap \G_m^+(1_{m-2})|\\&&+|\G_m^+(a^{-1}_{m-1})\cap \G_m^+(1_{m-2})|+|\G_m^-(a^{-1}_{m-1})\cap \G_m^+(1_{m-2})|\\&=&k .
\end{eqnarray*}
This completes the proof of the claim. \qed\\

Now we are ready to complete the proof of Case 2 in Lemma \ref{lem=ORR} by using the above claim. Let $\mathcal{A}:=\Aut(\G_m)$. As $|A([\G_m^+(1_1)])|\ne |A([\G_m^+(1_i)])|~(i\neq 1)$ by the claim, we have $[\G_m^+(1_1)]\not\cong [\G_m^+(1_i)]$ $(i\neq 1)$ and thus $\mathcal{A}$ fixes $G_1$ setwise. By $|A([\G_m^+(1_0)])|\neq |A([\G_m^+(1_2)])|$ (see the above claim), $\mathcal{A}$ fixes $G_i~(i=0,2)$ setwise. By $|T_{2,3}|=1$ and
$|T_{2,i}|=0$ $(i\neq 1,2,3)$, $\mathcal{A}$ fixes $G_3$ setwise. Similary, we obtain that $\mathcal{A}$ fixes $G_i~(i\in \mz_m)$ setwise. Recall that $R(G)$ is a group of automorphisms of $\G_m$ with $m$ orbits $G_i~(i\in \mz_m)$. Then $\mathcal{A}$ also has $m$ orbits satisfying $\mathcal{A}=R(G)\mathcal{A}_{1_0}$. Since $\mathcal{A}_{1_0}$ fixes $G_i$ setwise and $[G_0\cup G_1]\cong\G_2$ is an O$2$SR, $\mathcal{A}_{1_0}$ fixes $G_0$ and $G_1$ pointwise. As $T_{0,m-1}=\{1\}$, $\mathcal{A}_{1_0}$ fixes $G_{m-1}$ pointwise. Similarly,
we find that $\mathcal{A}_{1_0}$ fixes $G_{i}$ $(2\leq i\leq m-2)$ pointwise and hence $\mathcal{A}_{1_0}=1$. Thus, $\mathcal{A}=R(G)\mathcal{A}_{1_0}=R(G)$ holds and therefore $\G_m$ is an O$m$SR of $G$. This completes the proof of Case 2 in Lemma \ref{lem=ORR}.
\end{proof}

Now, we consider elementary abelian $2$-groups in two cases: of order at most $16$ (Lemma~\ref{lem=elementary abelian 2-group order le 16}); of order greater than $16$  (Lemma~\ref{lem=elementary abelian 2-group}). We need the following notation and remark on oriented digraphs with valency two.

\begin{notation}\label{defc}
Let $\Delta$ be an oriented digraph with valency two.
Define $\mathcal{C}(\Delta)$ by the set of oriented $3$-cycles in $\Delta$ and let $V(\mathcal{C}(\Delta))$ be the set of vertices in $\mathcal{C}(\Delta)$. For a subset $S\subseteq V(\Delta)$, define $\mathcal{C}_{S}(\Delta)$ by the set of cycles in $\mathcal{C}(\Delta)$ containing at least one vertex of $S$. For notational simplicity, put $\mathcal{C}_{s}(\Delta):=\mathcal{C}_{\{s\}}(\Delta)$. For a subset $W\subseteq \mathcal{C}(\Delta)$, define $N_{\Delta}(W)$ by the set of vertices and their neighbors of each cycle in $W$. Put
$N_{\Delta}(C):=N_{\Delta}(\{C\})$. If there are no confusions, we may omit $(\Delta)$ for each notation; $\mathcal{C}(\Delta)$, $V(\mathcal{C}(\Delta))$, $\mathcal{C}_{S}(\Delta)$, $\mathcal{C}_{s}(\Delta)$ and $N_{\Delta}(W)$.
\end{notation}

The following simple observation on oriented digraphs with valency two will be used in the proof of Lemma \ref{lem=elementary abelian 2-group order le 16}.

\begin{rem} \label{case1rmk}
Let $\Delta$ be an oriented digraph with valency two and let $\sigma \in \Aut(\Delta)$. Suppose that $\sigma $ fixes a vertex $v\in V(\Delta)$.
\begin{enumerate}
\item If $\sigma$ fixes an out-neighbor and an in-neighbor of $v$, then $\sigma$ fixes every neighbors of $v$ in $\Delta$.
\item Suppose that $\sigma$ fixes an oriented cycle $C\in \mathcal{C}(\Delta)$ pointwise.\\
\indent (i) Then $\sigma$ fixes $N_{\Delta}(C)$ pointwise. Moreover, if $s\in C$ then $\sigma$ fixes every vertex of $N_{\Delta}(\mathcal{C}_{s})$ in $\Delta$.\\
\indent (ii) Let $S\subseteq V(\Delta)$. If $C\cap S\ne \emptyset$ and cycles in $\mathcal{C}_S$ are connected to one another, then $\sigma$ fixes $N_{\Delta}(\mathcal{C}_S) $ pointwise.
\end{enumerate}
\end{rem}

\begin{lem}\label{lem=elementary abelian 2-group order le 16}
Let $m\geq 1$ be an integer. Each group $\mz_2^n~(0\leq n\leq 4)$ satisfies the following.
\begin{enumerate}
\item  $\mz_1$ does not admit O$m$SRs if and only if $2\leq m\leq 6$.
\item  $\mz_2$ does not admit O$m$SRs if and only if $m=2$.
\item  For each integer $2\leq n\leq 4$, $\mz_2^n$ does not admit O$m$SRs if and only if $m=1$ or $2$.
\end{enumerate}
\end{lem}

\begin{proof}
Let $G:=\mz_2^n~(0\leq n\leq 4)$. By Remark~\ref{rmkgdg} and Proposition~\ref{prop=ORR}, $G$ does not admit O$1$SRs if and only if $G=\mz_2^2,\mz_2^3$ or $\mz_2^4$. For the rest of the proof, we assume $m\geq 2$. We first show that
\begin{equation}\label{2-eleab}
\mbox{if~}(m,G)\in \{(2,\mz_1),(3,\mz_1),(4,\mz_1),(5,\mz_1),(6,\mz_1),(2,\mz_2),(2,\mz_2^2),
(2,\mz_2^3),(2,\mz_2^4) \} \mbox{ then $G$ does not admit O$m$SRs.}
\end{equation}
Since any vertex-transitive digraph cannot be O$m$SR ($m\geq 2$), digraphs with valency $0$, $1$ cannot be O$m$SR ($m\geq 2$). Moreover, digraphs of order less than $5$ cannot be O$m$SR. Hence $\mz_1$ does not admit O$m$SRs for $2\leq m\leq 4$, and $\mz_2$ does not admit O$2$SRs. With the aid of {\sc Magma} \cite{magma}, we can find that $G$ does not admit O$m$SRs for each $(m,G)\in \{(5,\mz_1),(6,\mz_1),(2,\mz_2^2), (2,\mz_2^3), (2,\mz_2^4)\}$. This shows statement (\ref{2-eleab}).\\

To complete the proof, we need to show that if $G$ satisfies $G=\mz_1$ with $m\geq 7$ or $G=\mz_2^n~(1\leq n\leq 4)$ with $m\geq 3$, then $G$ admits an O$m$SR. We divide the rest of the proof into five cases.\\

\f{\bf Case 1:} Show that $G=\mz_1$ admits an O$m$SR for any $m\geq 7$.\\
\noindent{\em Proof of Case 1. }For a given $m$, we will define a subset $T_{i,j}\subseteq G=\mz_1~(i,j\in \mz_m)$ and $\G_{m}:=\Cay(G,T_{i,j}:i,j \in \mz_m)$. Let $\mathcal{A}:=\Aut(\G_m)$, $\mathcal{C}:=\mathcal{C}(\G_m)$ and $\mathcal{C}_S:=\mathcal{C}_S(\G_m)$ for a subset $S\subseteq V(\G_m)$ (see Notation \ref{defc}).\\
For each $7\leq m\leq 11$, define a subset $T_{i,j}\subseteq G~(i,j\in \mz_m)$ as follows:
\[
\begin{array}{ll}
T_{0,3}=T_{1,4}=T_{2,1}=T_{m-1,2}=T_{j,j+2}=\{1\}~~& \mbox{ for } 3\leq j\leq m-2;\nonumber\\
T_{i,i+1}=\{1\}~~&\mbox{ for } i\in\mz_m;\nonumber\\
T_{i,j}=\emptyset~~&\mbox{ otherwise}.\nonumber
\end{array}
\]
Then $\G_m$ is an oriented $m$-Cayley digraph of $G$ with valency two. With the aid of {\sc Magma}~\cite{magma}, we find that $\G_m$ is an O$m$SR of $G$ for each $7\leq m\leq 11$.\\

On the other hand, consider Case 1 for $m\geq 12$. For each even integer $m\geq12$, define a subset $T_{i,j}\subseteq G~(i,j \in \mz_m)$ as follows:
\begin{eqnarray}
\{1\}&=&T_{0,1}=T_{1,2}=T_{2,3}=T_{3,4}=T_{4,5}=T_{5,6}=T_{6,0}~;\label{1st}\\
\{1\}&=&T_{9,10}=T_{10,12}=\cdots =T_{i,i+2}=\cdots=T_{m-4,m-2}=T_{m-2,m-1}\nonumber \\
&=&
T_{m-1,m-3}=\cdots=T_{i,i-2}=\cdots=T_{13,11}=T_{11,1}=T_{1,9}~; \label{2nd}\\
\{1\}&=&T_{0,4}=T_{4,2}=T_{2,6}=T_{6,7}=T_{7,8}=T_{8,3}=T_{3,5}=T_{5,7}=T_{7,m-2}=
T_{m-2,8}=T_{8,m-1}\nonumber\\
&=&T_{m-1,m-4}=T_{m-4,m-3}=\cdots=T_{i,i-3}=T_{i-3,i-2}=\cdots =
T_{13,10}=T_{10,11}=T_{11,9}=T_{9,0}~;\label{3rd}\\
\emptyset&=& T_{i,j}~ \mbox{ otherwise }.\nonumber
\end{eqnarray}

Then $\G_m$ is an oriented $m$-Cayley digraph of $G$ with valency two and $\G_m$ consists of three oriented cycles which correspond to (\ref{1st})-(\ref{3rd}), respectively (see the left one of {\sc Figure}~\ref{Fig2}).\\

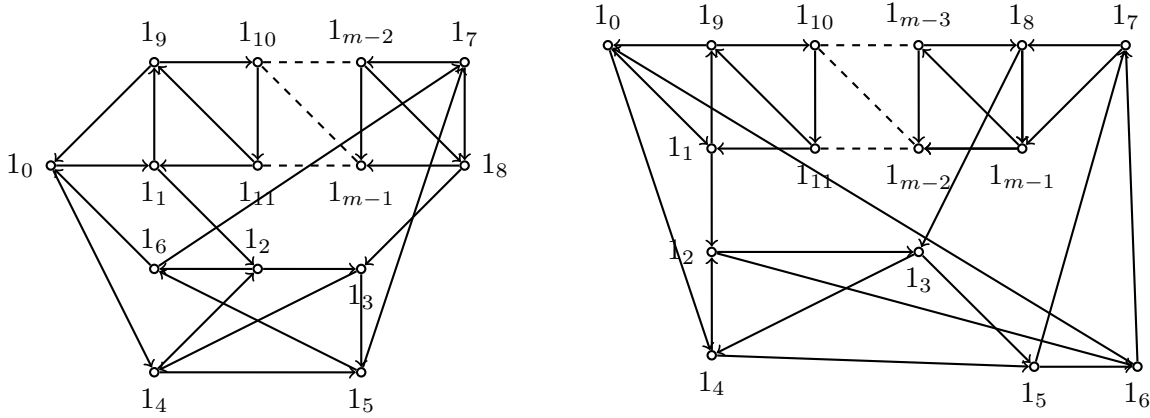
\begin{figure}[!ht]

\begin{tikzpicture}[node distance=1.1cm,thick,scale=0.7,every node/.style={transform shape},scale=1.6]
\node[circle](A0){};
\node[below=of A0, circle,draw, inner sep=1pt, label=left:{$1_0$}](A1){};
\node[right=of A1, circle,draw, inner sep=1pt, label=below:{$1_1$}](B1){};
\node[above=of B1, circle,draw, inner sep=1pt, label=above:{$1_9$}](B0){};
\node[below=of B1](B11){};

\node[right=of B1, circle,draw, inner sep=1pt, label=below:{$1_{11}$}](E1){};
\node[right=of B0, circle,draw, inner sep=1pt, label=above:{$1_{10}$}](E0){};

\node[right=of E1, circle,draw, inner sep=1pt, label=below:{$1_{m-1}$}](F1){};
\node[right=of E0, circle,draw, inner sep=1pt, label=above:{$1_{m-2}$}](F0){};

\node[right=of F1, circle,draw, inner sep=1pt, label=right:{$1_{8}$}](C1){};
\node[right=of F0, circle,draw, inner sep=1pt, label=above:{$1_{7}$}](C0){};
\node[below=of E1, circle,draw, inner sep=1pt, label=above:{$1_2$}](B2){};

\node[below=of C1](C11){};
\node[right=of B2, circle,draw, inner sep=1pt, label=below:{$1_3$}](C2){};
\node[below=of C2, circle,draw, inner sep=1pt, label=below:{$1_5$}](D2){};
\node[below=of B1, circle,draw, inner sep=1pt, label=above:{$1_6$}](D0){};
\node[below=of D0, circle,draw, inner sep=1pt, label=below:{$1_4$}](B3){};

\draw[->] (A1) to (B1);
\draw[->] (A1) to (B3);
\draw[->] (B1) to (B2);
\draw[->] (B1) to (B0);

\draw[->] (B2) to (C2);
\draw[->] (B2) to (D0);
\draw[->] (C2) to (B3);
\draw[->] (C2) to (D2);

\draw[->] (B3) to (D2);
\draw[->] (B3) to (B2);
\draw[->] (D2) to (D0);
\draw[->] (D2) to (C0);

\draw[->] (D0) to (C0);
\draw[->] (D0) to (A1);
\draw[->] (C0) to (C1);
\draw[->] (C0) to (F0);

\draw[->] (C1) to (C2);
\draw[->] (C1) to (F1);
\draw[->] (B0) to (E0);
\draw[->] (B0) to (A1);

\draw[->] (E1) to (B0);
\draw[->] (E1) to (B1);
\draw[->] (E0) to (E1);
\draw[->] (F0) to (F1);
\draw[->] (F0) to (C1);

\draw[dashed] (E0) to (F0);
\draw[dashed] (E1) to (F1);
\draw[dashed] (F1) to (E0);
\node[above=of D0 ](D00){};
\node[above=of D00 ](D01){};
\node[right=of D01](D02){};
\node[right=of D02](D04){};
\node[right=of D04](D03){};
\node[right=of D03, circle,draw, inner sep=1pt, label=above:{$1_0$}](A10){};
\node[right=of A10, circle,draw, inner sep=1pt, label=above:{$1_9$}](B10){};
\node[below=of B10, circle,draw, inner sep=1pt, label=left:{$1_1$}](B11){};
\node[below=of B11](B111){};
\node[below=of B11, circle,draw, inner sep=1pt, label=left:{$1_2$}](B12){};
\node[below=of B12, circle,draw, inner sep=1pt, label=below:{$1_4$}](B13){};

\node[right=of B10, circle,draw, inner sep=1pt, label=above:{$1_{10}$}](E10){};
\node[below=of E10, circle,draw, inner sep=1pt, label=below:{$1_{11}$}](E11){};
\node[right=of E10, circle,draw, inner sep=1pt, label=above:{$1_{m-3}$}](F10){};
\node[right=of E11, circle,draw, inner sep=1pt, label=below:{$1_{m-2}$}](F11){};

\node[right=of F10, circle,draw, inner sep=1pt, label=above:{$1_8$}](C10){};
\node[right=of C10, circle,draw, inner sep=1pt, label=above:{$1_{7}$}](D10){};
\node[below=of C10, circle,draw, inner sep=1pt, label=below:{$1_{m-1}$}](C11){};
\node[below=of C11](B122){};
\node[right=of B122](B123){};
\node[below=of B123, circle,draw, inner sep=1pt, label=below:{$1_6$}](D11){};
\node[left=of D11, circle,draw, inner sep=1pt, label=below:{$1_5$}](C13){};
\node[below=of F11, circle,draw, inner sep=1pt, label=below:{$1_3$}](D12){};
\draw[->] (A10) to (B11);
\draw[->] (A10) to (B13);
\draw[->] (B11) to (B12);
\draw[->] (B11) to (B10);
\draw[->] (B12) to (D12);
\draw[->] (B12) to (D11);
\draw[->] (D12) to (B13);
\draw[->] (D12) to (C13);
\draw[->] (B13) to (B12);
\draw[->] (B13) to (C13);
\draw[->] (C13) to (D11);
\draw[->] (C13) to (D10);
\draw[->] (D11) to (D10);
\draw[->] (D11) to (A10);
\draw[->] (D10) to (C10);
\draw[->] (D10) to (C11);
\draw[->] (C10) to (C11);
\draw[->] (C10) to (D12);
\draw[->] (B10) to (E10);
\draw[->] (B10) to (A10);
\draw[->] (C11) to (F11);

\draw[->] (E11) to (B11);
\draw[->] (E11) to (B10);
\draw[->] (E10) to (E11);

\draw[->] (F10) to (C10);
\draw[->] (C11) to (F10);
\draw[->] (F10) to (F11);
\draw[->] (C10) to (C11);
\draw[->] (C11) to (F11);

\draw[dashed] (E10) to (F10);
\draw[dashed] (E11) to (F11);
\draw[dashed] (E10) to (F11);
\end{tikzpicture}
\caption{O$m$SR of $\mathbb{Z}_1$ for $m\geq12$}\label{Fig2}
\end{figure}
We first show $\mathcal{A}_{1_0}=1$ by using {\sc Figure} \ref{Fig2} and Remark \ref{case1rmk}. By $|\mathcal{C}_{1_0}|=1$ and $\mathcal{C}_{1_0}=[\{1_0,1_1,1_9\}]$, $\mathcal{A}_{1_0}$ fixes cycle $\mathcal{C}_{1_0}$.\\
By Remark~\ref{case1rmk} with $(\sigma, s, S)=(\mathcal{A}_{1_0}, 1_9, \{1_9, 1_{10}, 1_{12}, \cdots,  1_{m-1}\}) $, $\mathcal{A}_{1_0}$ fixes $N(\mathcal{C}_{\{1_9, 1_{10}, 1_{12}, \cdots,  1_{m-1}\}})$ pointwise. On the other hand, we find that $|\mathcal{C}_{1_4}|=1$ and two cycles $\mathcal{C}_{1_0}$ and $\mathcal{C}_{1_4}=[\{1_4,1_2,1_3\}]$ are adjacent. This implies that $\mathcal{A}_{1_0}$ fixes $\mathcal{C}_{1_4}$ pointwise, and also $N(\mathcal{C}_{\{1_9, 1_{10}, 1_{12}, \cdots,  1_{m-1}\}}\cup \mathcal{C}_{1_4})$. Now  $\mathcal{A}_{1_0}=1$ follows as $N(\mathcal{C}_{\{1_9, 1_{10}, 1_{12}, \cdots,  1_{m-1}\}}\cup \mathcal{C}_{1_4})=V(\G_m)$. By the left one of {\sc Figure}~\ref{Fig2}, it is easy to find that $\{1_5,1_6,1_7,1_8\}=V(\G_m)\setminus V(\mathcal{C})$ and $1_5\in N(\mathcal{C})$ both hold. This implies that $\mathcal{A}$ fixes $1_5$, and $\mathcal{A}$ fixes directed path $(1_5,1_6,1_7)$ pointwise because the induced digraph $[\{1_5,1_6,1_7,1_8\}]$ is the union of directed path $(1_5,1_6,1_7,1_8)$ and arc $(1_5,1_7)$. Since $1_0$ is a neighbor of $1_6$, $\mathcal{A}$ fixes $1_0$ by Remark \ref{case1rmk}. As $\mathcal{A}_{1_0}=1$, $\mathcal{A}$ fixes each vertex of $\G_m$. Now Case 1 holds for even integer $m\geq 12$.\\

Finally, we consider Case 1 for odd integer $m>12$. Define a subset $T_{i,j}\subseteq G~(i,j \in \mz_m)$ as follows:
\begin{eqnarray}
\{1\}&=& T_{0,1}=T_{1,2}=T_{2,3}=T_{3,4}=T_{4,5}=T_{5,6}=T_{6,0}~;\label{1stt}\\
\{1\}&=&T_{9,10}=T_{10,12}=\cdots=T_{i,i+2}=\cdots=T_{m-5,m-3}=T_{m-3,8}=
T_{8,m-1}=T_{m-1,m-2}\nonumber \\
&=&T_{m-2,m-4}=\cdots=T_{i,i-2}=\cdots=T_{13,11}=T_{11,1}=T_{1,9}~;\label{2ndd}\\
\{1\}&=&T_{0,4}=T_{4,2}=T_{2,6}=T_{6,7}=T_{7,8}=T_{8,3}=T_{3,5}=T_{5,7}=T_{7,m-1}=
T_{m-1,m-3}=T_{m-3,m-2}=T_{m-2,m-5}\nonumber \\ &=&\cdots =T_{i,i+1}=T_{i+1,i-2}=\cdots =
T_{12,13}=T_{13,10}=T_{10,11}=T_{11,9}=T_{9,0}~;\label{3rdd}\\
\emptyset &=&T_{i,j}~\mbox{ otherwise~}.\nonumber
\end{eqnarray}
Then $\G_m$ is an oriented $m$-Cayley digraph of $G$ with valency two, and $\G_m$ consists of three oriented cycles which correspond to (\ref{1stt})-(\ref{3rdd}), respectively (see the right one of {\sc Figure}~\ref{Fig2}). With $S:=\{1_9, 1_{10}, 1_{12}\cdots 1_{m-3}\}$, we can obtain $\mathcal{A}_{1_0}=1$ in the same way as Case 1 for even integer $m\geq 12$. By the right one of {\sc Figure}~\ref{Fig2}, it is easy to find that $\{1_5,1_6,1_7\}=V(\G_m)\setminus V(\mathcal{C})$ and $1_5\in N(\mathcal{C})$ both hold. This implies that $\mathcal{A}$ fixes $1_5$, and $\mathcal{A}$ fixes directed path $(1_5,1_6,1_7)$ pointwise because induced digraph $[\{1_5,1_6,1_7\}]$ is the union of directed path $(1_5,1_6,1_7)$ and arc $(1_5,1_7)$. Since $1_0$ is a neighbor of $1_6$, $\mathcal{A}$ fixes $1_0$ by Remark \ref{case1rmk}. As $\mathcal{A}_{1_0}=1$, $\mathcal{A}$ fixes each vertex of $\G_m$. Now Case 1 holds for odd integer $m> 12$. This completes the proof of Case 1.\qed \\

For the remaining cases we need to use the following statement (\ref{sigmaclaim}) for new graph $\Delta$ from $\G_m$, where $\G_m$ is the graph defined in Case 1. For a positive integer $m\geq 12$ and a subset $\mathbb{E}\subseteq \{(1_6$, $1_0)$, $(1_1,1_2)$, $(1_8,1_3)$, $(1_4,1_5)\}$, let $\Delta$ be a connected oriented digraph with valency two containing $\G_m-\mathbb{E}$ as a subgraph, where $\G_m-\mathbb{E}$ is the subgraph of $\G_m$ by deleting all arcs in $\mathbb{E}$. Note that the graph $\G_m-\mathbb{E}$ is connected and it contains every oriented $3$-cycles in $\mathcal{C}(\G_m)\cup \mathcal{C}_{1_4}(\G_m)$ and arc $(1_0,1_4)$, where $\mathcal{C}_{1_4}(\G_m)=[ \{1_4,1_2,1_3\}]$. For each $e\in \mathbb{E}$, we denote by $v(e)$ the set of two vertices incident to $e$ and let $\Aut(\Delta)_{(v(e))}$ be the subgroup of $\Aut(\Delta)$ fixing $v(e)$ pointwise. We now show that the following holds.
\begin{equation}\label{sigmaclaim}
\mbox{Each~group~in~}\{\Aut(\Delta)_{1_0}, \Aut(\Delta)_{1_5}, \Aut(\Delta)_{(v(e))}\ |\ e\in \mathbb{E}\} \mbox{~fixes~each vertex~of~}\G_m.
\end{equation}
First, we can show that $\Aut(\Delta)_{1_0}$ fixes each vertex of $\G_m$ in the same way with Case 1. Let $e\in \mathbb{E}$. If $e=(1_6,1_0)$ then $\Aut(\Delta)_{(v(e))}$ fixes $1_0$. If $e\ne(1_6,1_0)$, then $e$ and $\mathcal{C}_{1_4}(\G_m)$ has a common vertex. By Remark \ref{case1rmk}, we obtain that $\Aut(\Delta)_{(v(e))}$ fixes $\mathcal{C}_{1_4}(\G_m)=[\{1_4,1_2,1_3\}]$ pointwise, and also fixes each neighbors of $1_4$. Since $1_0$ is a neighbor of $1_4$, $\Aut(\Delta)_{(v(e))}$ fixes $1_0$ and hence $\Aut(\Delta)_{(v(e))}$ fixes each vertex of $\G_m$. Now consider $\Aut(\Delta)_{1_5}$. Since two out-neighbors of $1_5$ in $\Delta$ are $1_6$ and $1_7$, and $(1_6,1_7)$ is an arc, we obtain that $\Aut(\Delta)_{1_5}$ fixes cycle $[\{1_5,1_6,1_7\}]$ pointwise. As $1_2$ is adjacent to $1_6$, $\Aut(\Delta)_{1_5}$ fixes $1_2$ by Remark \ref{case1rmk}. Thus, $\Aut(\Delta)_{1_5}$ fixes cycle
$\mathcal{C}_{1_4}(\G_m)=[\{1_4,1_2,1_3\}]$ pointwise, and thus it fixes each vertex of $\G_m$. Now (\ref{sigmaclaim}) follows and we are ready to use it.\\

For the remaining cases Case 2--Case 5, we will define a subset $T_{i,j}\subseteq G~(i,j\in \mz_m)$ for each $m$ and $G$, and let $\Sigma_{m}:=\Cay(G,T_{i,j}:i,j \in \mz_m)$ and $\mathcal{A}:=\Aut(\Sigma_m)$. Then we can check that each $\Sigma_m$ is an oriented $m$-Cayley digraph of $G$ with valency two.\\

\f{\bf Case 2:} Show that $G=\langle x\rangle= \mz_2$ admits an O$m$SR for any $m\geq 3$.\\
\noindent{\em Proof of Case 2. }For each $3\leq m\leq 11$, define a subset $T_{i,j}\subseteq G~(i,j \in \mz_m)$ as follows. Let $T_{i,j}=\emptyset ~(i,j \in \mz_m)$ except
\[\begin{array}{ll}
\mbox{for~}m=3 &T_{0,1}=T_{0,2}=T_{2,0}=\{1\},~T_{1,0}=T_{1,2}=T_{2,1}=\{x\}~;\\
\mbox{for~}m=4 & T_{0,1}=T_{0,2}=T_{3,0}=T_{3,1}=\{1\},~T_{1,0}=T_{1,2}=\{x\},~T_{2,3}=\{1,x\}~; \\
\mbox{for~}m=5 &T_{0,1}=T_{0,3}=T_{2,0}=T_{2,4}=T_{3,1}=T_{3,4}=T_{4,3}=T_{4,2}=\{1\},~T_{1,0}=T_{1,2}=\{x\}~;\\
\mbox{for~}m=6 &T_{0,1}=T_{0,3}=T_{2,0}=T_{2,5}=T_{3,1}=T_{3,4}=T_{4,3}=T_{5,4}=\{1\},~T_{1,0}=T_{1,2}=T_{4,5}=T_{5,2}=\{x\}~;\\
\mbox{for~}7\leq m\leq11&T_{0,1}=T_{0,3}=T_{2,0}=T_{2,5}=T_{3,1}=T_{3,4}=T_{4,3}=T_{5,6}=T_{i,i+1}=T_{m-1,4}=\{1\}, \\ &T_{1,0}=T_{1,2}=T_{4,m-1}=T_{5,2}=T_{i,i-1}=T_{m-1,m-2}=\{x\}~~(6\leq i\leq m-2).
\end{array}\]
With the aid of {\sc Magma}~\cite{magma}, we find that $\Sigma_m$ is an O$m$SR of $G$ for each $3\leq m\leq 11$.\\

For $m\geq 12$, take $T_{i,j}~(i,j\in \mz_m)$ the same as Case~1 except $T_{6,0}=\{x\}$. Note that we denote by $\G_m$ ($\Sigma_m$, respectively) the $m$-Cayley digraph defined in Case~1 (Case 2-Case 5, respectively). Clearly, $\Sigma_m$ has two layers: vertex sets $\{1_i\ |\ i\in \mz_m\}$ and $\{x_i\ |\ i\in \mz_m\}$.  They are called {\em $1$-layer} and {\em $x$-layer of $\Sigma_m$}, respectively (this notation {\em $i$-layer} will be used in Case 2-Case 5). For each $g\in \{1,x\}$, the induced subgraph of $g$-layer is isomorphic to the subgraph $\G_m-\{(1_6,1_0)\}$ of $\G_m$. By (\ref{sigmaclaim}), $\mathcal{A}_{1_0}$ fixes $1$-layer pointwise. Since $1$-layer has only two neighbors in $x$-layer (say, $x_0$ and $x_6$), $\mathcal{A}_{1_0}$ fixes $x_0$ and $x_6$, respectively. As $\mathcal{A}_{(\{x_0,x_6\})}$ fixes each vertex of $x$-layer, $\mathcal{A}_{1_0}$ fixes every vertex of $\Sigma_m$ (i.e., $\mathcal{A}_{1_0}=1$). We also obtain $\mathcal{A}_{1_5}=\mathcal{A}_{1_0}=1$ by (\ref{sigmaclaim}). In view of {\sc Figure}~\ref{Fig2}, we can check that for even $m\geq 12$, the vertices not on any oriented $3$-cycle of $\Sigma_m$ are the vertices in $G_5\cup G_6\cup G_7\cup G_8$, and only $G_5$ has two in-neighbors on an oriented $3$-cycle; for odd  $m>12$, the vertices not on any oriented $3$-cycle of $\Sigma_m$ are the vertices in $G_5\cup G_6\cup G_7$, and only $G_5$ has two in-neighbors on an oriented $3$-cycle. Thus, $\mathcal{A}$ fixes $G_5$ setwise. Since $R(G)$ is transitive on $G_5$, we have $\mathcal{A}=R(G)\mathcal{A}_{1_5}=R(G)$. Hence $\Sigma_m$ is an O$m$SR of $G$. This completes the proof of Case 2.\qed \\

\f{\bf Case 3:} Show that $G=\langle x,y\rangle= \mz_2^2$ admits an O$m$SR for any $m\geq 3$.\\
\noindent{\em Proof of Case 3. }For each $3\leq m\leq 11$, define a subset $T_{i,j}\subseteq G~(i,j \in \mz_m)$ as follows. Let $T_{i,j}=\emptyset ~(i,j \in \mz_m)$ except
\[\begin{array}{ll}
\mbox{for~}m=3 &T_{0,1}=T_{1,2}=\{1\},~T_{2,1}=\{x\},~T_{1,0}=T_{2,0}=\{y\},~T_{0,2}=\{xy\}~;\\
\mbox{for~}m=4 &T_{0,1}=T_{1,3}=T_{3,2}=\{1\},~T_{2,3}=T_{3,1}=\{x\},~T_{1,0}=T_{2,0}=\{y\},~T_{0,2}=\{xy\}~;\\
\mbox{for~}5\leq m\leq 11& T_{0,1}=T_{1,3}=T_{3,4}=T_{i,i+1}=T_{m-1,2}=\{1\},~T_{1,0}=T_{2,0}=\{y\},~T_{0,2}=\{xy\},\\ &T_{2,m-1}=T_{3,1}=T_{i,i-1}=T_{m-1,m-2}=\{x\}~(4\leq i\leq m-2).
\end{array}\]
With the aid of {\sc Magma}~\cite{magma}, we find that $\Sigma_m$ is an O$m$SR of $G$ for each $3\leq m\leq 11$.\\
For $m\geq 12$,  define a subset $T_{i,j}\subseteq G~(i,j \in \mz_m)$ the same as Case~1 except
$$T_{6,0}=\{x\} \mbox{~and~} T_{1,2}=\{y\}.$$

Then $\Sigma_m$ has four layers. The induced subgraph of each layer is isomorphic to the induced subgraph of $1$-layer, which is exactly the subgraph $\G_m-\{(1_6,1_0),(1_1,1_2)\}$ of $\G_m$. By (\ref{sigmaclaim}), $\mathcal{A}_{1_0}$ fixes $1$-layer pointwise, and also it fixes each vertex of $\{x_0,x_6,y_1,y_2\}$. Again by (\ref{sigmaclaim}), $\mathcal{A}_{(\{x_0,y_6\})}$ fixes $x$-layer pointwise; $\mathcal{A}_{(\{ y_1,y_2\})}$ fixes $y$-layer pointwise; $\mathcal{A}_{(\{ y_1,y_2\})}$ fixes $(xy)_0$ and $(xy)_6$ because they are adjacent to $y_6$ and $y_0$, respectively. Therefore $\mathcal{A}_{(\{(xy)_0,(xy)_6\})}$ fixes $xy$-layer pointwise, and hence $\mathcal{A}_{1_0}$ fixes every vertex of $\Sigma_m$ (i.e., $\mathcal{A}_{1_0}=1$). Thus, $\mathcal{A}_{1_5}=\mathcal{A}_{1_0}=1$. In the same way as Case 2, we can show that $\mathcal{A}$ fixes $G_5$ setwise, and $\Sigma_m$ is an O$m$SR of $G$. This completes the proof of Case 3.\qed\\

\f{\bf Case 4:} Show that $G=\langle x,y,z \rangle= \mz_2^3$ admits an O$m$SR for any $m\geq 3$.\\
\noindent{\em Proof of Case 4. }For each $3\leq m\leq 11$, define a subset $T_{i,j}\subseteq G~(i,j \in \mz_m)$ as follows. Let $T_{i,j}=\emptyset ~(i,j \in \mz_m)$ except
\[\begin{array}{ll}
\mbox{for~}m=3 &T_{0,1}=T_{1,2}=\{1\},~T_{2,1}=\{x\},~T_{2,0}=\{y\},~T_{1,0}=T_{0,2}=\{z\}~;\\
\mbox{for~}m=4 &T_{0,1}=T_{1,3}=T_{3,2}=\{1\},\,\,\,\,\,\,T_{2,3}=T_{3,1}=\{x\},\,\,\,\,\,\, T_{2,0}=\{y\},~T_{0,2}=\{z\}~;\\
\mbox{for~}5\leq m\leq 11 &T_{0,1}=T_{1,3}=T_{3,4}=T_{i,i+1}=T_{m-1,2}=\{1\},~T_{2,0}=\{y\},~T_{1,0}=T_{0,2}=\{z\},\\ &T_{2,m-1}=T_{3,1}=T_{i,i-1}=T_{m-1,m-2}=\{x\}~(4\leq i\leq m-2).
\end{array}\]
With the aid of {\sc Magma}~\cite{magma}, we find that $\Sigma_m$ is an O$m$SR of $G$ for each $3\leq m\leq 11$.\\
For $m\geq 12$, define a subset $T_{i,j}\subseteq G~(i,j \in \mz_m)$ the same as Case~1 except
$$T_{6,0}=\{x\},~T_{1,2}=\{y\},~T_{8,3}=\{z\}.$$
In the same way as Case 3, we can show that $\Sigma_{m}$ is an O$m$SR of $G$. This completes the proof of Case 4. \qed \\

\f{\bf Case 5:} Show that $G=\langle x,y,z,w\rangle= \mz_2^4$ admits an O$m$SR for any $m\geq 3$.\\
\noindent{\em Proof of Case 5. }For each $3\leq m\leq 11$, define a subset $T_{i,j}\subseteq G~(i,j \in \mz_m)$ as follows. Let $T_{i,j}=\emptyset ~(i,j \in \mz_m)$ except
\[\begin{array}{ll}
\mbox{for~}m=3~&T_{0,1}=\{1,y,xy\},~T_{1,2}=\{1,z,w\},~T_{2,0}=\{y,w,xw\}~;\\
\mbox{for~}4\leq m\leq 11~&T_{0,1}=\{1,y,xy\},~T_{1,2}=\{1,z,w\},~T_{m-1,0}=\{y,w,xw\},~T_{i,i+1}=\{x,y,w\} ~(2\leq i\leq m-2).
\end{array}\]
With the aid of {\sc Magma}~\cite{magma}, we find that $\Sigma_m$ is an O$m$SR of $G$ for each $3\leq m\leq 11$. \\
For $m\geq 12$, define a subset $T_{i,j}\subseteq G~(i,j \in \mz_m)$ the same as Case~1 except
$$T_{6,0}=\{x\},~T_{1,2}=\{y\},~T_{8,3}=\{z\},~T_{4,5}=\{w\}.$$
In the same way as Case 3, we can show that $\Sigma_{m}$ is an O$m$SR of $G$. This completes the proof of Case 5. Now Lemma \ref{lem=elementary abelian 2-group order le 16} follows.
\end{proof}

\begin{rem}
The oriented $m$-Cayley digraphs constructed in Lemma~\ref{lem=elementary abelian 2-group order le 16} have valency two. In particular, $\mz_1$ admits an O$m$SR with valency two if and only if $m\geq 7$, and $\mz_2^n$ ($1\leq n\leq 4$) admits an O$m$SR with valency two if and only if $m\geq 3$.
\end{rem}

\begin{lem}\label{lem=elementary abelian 2-group}
For any integers $n\geq 5$ and $m\geq 2$, $\mz_2^n$ admits an O$m$SR.
\end{lem}

\begin{proof}
For a given integer $n\geq 5$, let $G:=\mz_2^n=\langle x_1,x_2,\cdots,x_n\rangle$, $x:=x_1x_2\cdots x_n$ and $\overline{x_i}:=x x_i$ $(1\leq i\leq n)$. Take subsets $S,R,T\subseteq G$ as follows:
$$S=\{1,\,x_i~|~1\leq i\leq n\},~R=\{x,\,\overline{x_i}~|~1\leq i\leq n\},~T=\{x_1x_2x_{n-2}x_{n-1},\,x_1x_2x_{n-1}x_n,\,x_ix_{i+1}~|~1\leq i\leq n-1\}.$$
It is easy to see that $R=xS$ and $|R^2|=|S^2|=|\{1,\,x_i,\,x_ix_j\ |\ 1\leq i,j\leq n\}|=1+n+\frac{n(n-1)}{2}$ hold. Since set $\{x_ix_{i+1}x_{i+2}~|~1\leq i\leq n-2\}$ appears twice in $ST$ where
$$ST=T\cup \{x_i,\,x_kx_jx_{j+1}~|~1\leq i,j,k \leq n,\, j\ne n,\,k\neq j,j+1\}\cup \{x_1x_2x_ix_{n-2}x_{n-1},\,x_1x_2x_jx_{n-1}x_n~|~3\leq i,j\leq n-2, i\ne  n-2\},$$ we obtain
$$|ST|=n+1+n+(n-1)(n-3)+(n-5)+(n-4)-(n-2)=n^2-n-3.$$ By $x_1x_2\in ST$, $x_1x_2\not\in SR$, $\overline{x_1}\in SR$ and $\overline{x_1}\not\in ST$, we have the following :
\begin{eqnarray}\label{abeEq2}
|R^2|=|S^2|=|SR|,~|RT|=|ST|>|R^2|,~ST\nsubseteq SR,~SR\nsubseteq ST.
\end{eqnarray}
We divide the proof into two cases: $m=2$ (Case 1) and $m\geq 3$ (Case 2).\\

\f{\bf Case 1: }Show that $G$ admits an O$2$SR.\\
\noindent{\em Proof of Case 1. }Let $\G_2:=\BiCay(G,\emptyset,\emptyset,S,T)$ and $\mathcal{A}:=\Aut(\G_2)$. Then $\G_2$ is an oriented $2$-Cayley digraph of $G$ with valency $n+1$. If $n=5$, then we find that $\G_2$ is an O$2$SR of $G$ by using {\sc Magma}~\cite{magma}. Let $n\geq 6$. Note that
\[\G_2^+(1_0)=S_1=\{1_1,\,(x_i)_1~|~1\leq i\leq n\},~x_iT=\{x_ix_jx_{j+1},\,x_ix_1x_2x_{n-2}x_{n-1},\,x_ix_1x_2x_{n-1}x_n~|~1\leq j<n\}~(1\leq i\leq n)\]
all hold. Since $R(G)\in \mathcal{A}$ and $G$ is abelian, we have
$$\G_2^+(1_1)=T_0~\mbox{ and } \G_2^+((x_i)_1)=(x_iT)_0~~(1\leq i\leq n).$$
It follows by $T\cap x_iT=\emptyset~(1\leq i\leq n)$ and $x_iT\cap x_{i+2}T=\{x_{i+1},\,x_{i}x_{i+1}x_{i+2}\}~~ (1\leq i\leq n-2)$ that
\begin{eqnarray}\label{abeEq1}
\G_2^+((x_i)_1)\cap\G_2^+(1_1)=\emptyset~(1\leq i\leq n)\mbox{ and } \G_2^+((x_i)_1)\cap\G_2^+((x_{i+2})_1)=\{x_{i+1},\,x_{i}x_{i+1}x_{i+2}\}_0~~(1\leq i\leq n-2).
\end{eqnarray}
This shows that $1_1$ is the only vertex in $\G_2^+(1_0)$ that has no common out-neighbor with any other vertex in $\G_2^+(1_0)$, and $(x_i)_1~(1\leq i\leq n)$ has at least two common out-neighbors with some vertex in $\G_2^+(1_0)$. Hence $\mathcal{A}_{1_0}$ fixes $1_1$ and thus $\mathcal{A}_{1_0}\subseteq \mathcal{A}_{1_1}$. Similarly, the out-neighbors of $1_1$ can be written as
\begin{eqnarray*}
\G_2^+(1_1)=T_0=\{(x_1x_2x_{n-2}x_{n-1})_0,\,(x_1x_2x_{n-1}x_n)_0,(x_jx_{j+1})_0~|~1\leq j\leq n-1\}.
\end{eqnarray*}
We also find $\G_2^+(t_0)=(tS)_1~(t_0\in T_0)$ and the following:
\[\begin{array}{ll}
x_jx_{j+1}S\cap x_{j+1}x_{j+2}S&=\{x_{j+1},\,x_jx_{j+1}x_{j+2}\}~(1\leq j\leq n-2),~\\
x_1x_2x_{n-2}x_{n-1}S\cap x_1x_2S&=\{x_1x_2x_{n-2},\,x_1x_2x_{n-1}\}, ~\\
x_1x_2x_{n-1}x_{n}S\cap x_{n-1}x_nS&=\{x_1x_{n-1}x_n,\,x_2x_{n-1}x_n\}.
\end{array}\]
This shows that each vertex in $\G_2^+(1_1)$ has at least two common out-neighbors with some vertex in $\G_2^+(1_1)$. As $\G_2$ cannot be vertex transitive by (\ref{abeEq1}), $\mathcal{A}$ fixes $G_0$ and $G_1$ setwise. As $R(G)\subseteq \mathcal{A}$, $\mathcal{A}$ has exactly two orbits $G_0$ and $G_1$ and hence $\mathcal{A}=R(G)\mathcal{A}_{1_0}=R(G)\mathcal{A}_{1_1}$. Since $R(G)$ is semiregular and $\mathcal{A}_{1_0}\subseteq \mathcal{A}_{1_1}$, $\mathcal{A}_{1_1}=\mathcal{A}_{1_1}\cap \mathcal{A}=\mathcal{A}_{1_1}\cap R(G)\mathcal{A}_{1_0}=\mathcal{A}_{1_0}$ holds and so $\mathcal{A}_{g_0}=\mathcal{A}_{g_1}$ $(g\in G)$. Let $\sigma\in \mathcal{A}$. Suppose $\{g_0,g_1\}^{\sigma}\cap \{g_0,g_1\}\not=\emptyset$. Then it follows by $\mathcal{A}_{g_0}=\mathcal{A}_{g_1}$ and $\mathcal{A}$ fixes $G_i~(i=0,1)$ that
$$g_0^{\sigma}=g_0 \mbox{~and ~ }g_1^{\sigma}=g_1$$
both hold. This means that $\{\{g_0,g_1\}\ |\ g\in G\}$ is a block system of $\mathcal{A}$ on $V(\G_2)$, and hence there is a permutation $\overline{\sigma}$ of $G$ satisfying $(g^{\overline{\sigma}})_0=(g_0)^{\sigma}$ and $(g^{\overline{\sigma}})_1=(g_1)^{\sigma}$ for each $g\in G$. By Proposition~\ref{prop=Haar}, $\overline{\sigma}\in \Aut(\Cay(G,T))$. Since $\Cay(G,T)$ is a GRR of $G$ by \cite[pp.654]{Imrich1}, we obtain $\overline{\sigma} \in R(G)$. This implies $\sigma \in R(G)$ and hence $\mathcal{A}=R(G)$. Therefore, $\G_2$ is an O$2$SR of $G$. This completes the proof of Case 1. \qed \\

\f{\bf Case 2:} Show that $G$ admits an O$m$SR for any $m\geq 3$.\\
\noindent{\em Proof of Case 2. }Let $m\geq 3$ and take a subset $T_{i,j}\subseteq G~(i,j \in \mz_m)$ as follows:
\begin{eqnarray*}
T_{1,0}=T,~T_{0,1}=T_{i,i+1}=S,~T_{i,i-1}= R,~T_{i,j}=\emptyset \mbox{ for } i\ne 1,~j\ne i\pm1.
\end{eqnarray*}
Let $\G_m:=\Cay(G,T_{i,j}:i,j \in \mz_m)$ and $\mathcal{A}:=\Aut(\G_m)$. Then $\G_m$ is an oriented $m$-Cayley digraph of $G$ with valency $2n+2$ (See {\sc Figure}~\ref{Fig3}). It is easy to see that $\G_m^+(1_1)=S_2\cup T_0$ and $\G_m^+(1_i)=S_{i+1}\cup R_{i-1}$ for each $i\neq 1$. Recall that $\G_m^+(\G_m^+(1_i))$ is the set of out-neighbors of $\G_m^+(1_i)$ (See Notation~\ref{arcdef}). Since $G$ is abelian, we obtain
\begin{eqnarray*}
\G_m^+(\G_m^+(1_0))&=&\G_m^+(S_1\cup R_{m-1})=(S^2)_2\cup(ST)_0\cup (SR)_0\cup(R^2)_{m-2}~;\\
\G_m^+(\G_m^+(1_1))&=&\G_m^+(S_2\cup T_0)=(S^2)_3\cup(SR)_1 \cup(ST)_1\cup(TR)_{m-1}~;\\
 \G_m^+(\G_m^+(1_2))&=&\G_m^+(S_{3}\cup R_{1})=(S^2)_{4}\cup(SR)_2\cup(RT)_{0}~;\\
\G_m^+(\G_m^+(1_i))&=&\G_m^+(S_{i+1}\cup R_{i-1})=(S^2)_{i+2}\cup(SR)_i \cup(R^2)_{i-2}~(3\leq i\leq m-1).
\end{eqnarray*}
\begin{figure}
  \centering
  \includegraphics[width=9cm]{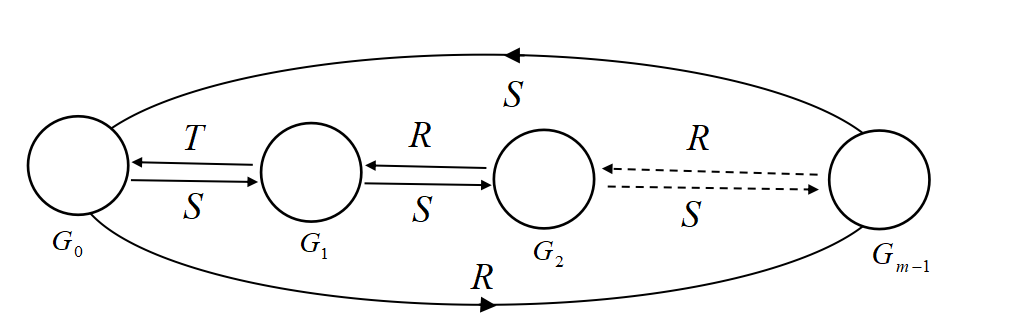}
\caption{The digraph $\G_m$~for $m\geq 3$}\label{Fig3}
\end{figure}
By (\ref{abeEq2}), it is easy to find the following:
\[\begin{array}{ll}
|\G_m^+(\G_m^+(1_0))|> |\G_m^+(\G_m^+(1_i))|, ~~~~&|\G_m^+(\G_m^+(1_1))|>|\G_m^+(\G_m^+(1_i))|~(3\leq i\leq m-1)~;\\
|\G_m^+(\G_m^+(1_1))|>|\G_m^+(\G_m^+(1_2))|,~~~~&|\G_m^+(\G_m^+(1_0))|>|\G_m^+(\G_m^+(1_2))|.
\end{array}
\]
So $\mathcal{A}$ fixes $G_0\cup G_1$ setwise. Since $[G_0\cup G_1]\cong \G_2$ is an O$2$SR of $G$, $\mathcal{A}$ fixes $G_0\cup G_1$ pointwise. Since all out-neighbors of $G_1$ are in $G_0\cup G_2$, $\mathcal{A}$ fixes $G_2$ setwise. Take $g,h\in G$ with $g\ne h$. Then $$\G_m^-(g_2)\cap G_1=(gS)_1=\{g,x_ig~|~1\leq i\leq n\}_1 \mbox{~and ~} \G_m^-(h_2)\cap G_1=(hS)_1=\{h,x_ih~|~1\leq i\leq n\}_1.$$ As $gS\ne hS$, $\mathcal{A}$ fixes $G_2$ pointwise. In the same way, we can show that $\mathcal{A}$ fixes $G_i$ pointwise $(3\leq i\leq m-1)$. Therefore, $\mathcal{A}_{1_1}=1$ and thus $\mathcal{A}=R(G)\mathcal{A}_{1_1}=R(G)$ (i.e., $\G_m$ is an O$m$SR of $G$). This completes the proof for Case 2 of Lemma~ \ref{lem=elementary abelian 2-group}.
\end{proof}

Now we consider finite non-abelian generalized dihedral groups.
\begin{lem}\label{lem=generalizeddihedral}
Let $G$ be a finite non-abelian generalized dihedral group. Then $G$ admits an O$m$SR for any integer $m\geq 2$.
\end{lem}

\begin{proof}
In view of Remark~\ref{rmkgdg}, take an abelian group $H$ of exponent greater than $2$ and an element $b\in G$ satisfying $G=\langle H,b\rangle$, $o(b)=2$ and $bhb=h^{-1}$ for each $h\in H$, where $o(b)$ is the order of $b$ in $G$. Now, we divide the proof into two cases.\\

\f{\bf Case 1:} $H$ admits an ORR.\\
\noindent{\em Proof of Case 1. } Let $R$ be a subset of $H$ such that
$\Sigma:=\Cay(H,R)$ is an ORR of $H$.
If $R=\emptyset$, then $H$ is $\mz_1$ or $\mz_2$ and thus $G$ is $\mz_2$ or $\mz_2^2$. This is impossible as $G$ is non-abelian.
Thus, $R\ne \emptyset $ and take $a\in R$. Since $b\not\in H$, note that
\begin{eqnarray}\label{Eq18}
o(a)\geq 3,\ \ a\in R,\ \ a\not\in R^{-1},\ \ 1\not\in R,\ \ b\not\in R\cup R^{-1}, ~H=\langle R\rangle,~ R\cap R^{-1}=\emptyset
\end{eqnarray}
all hold. Put $k:=|A([\Sigma^+(1)])|$. Now, we divide Case 1 into two subcases: $m=2$ (Subcase 1.1) and $m\geq 3$ (Subcase 1.2).\\

\f{\bf Subcase 1.1: }Show that $G$ admits an O$2$SR. \\
Let $\G_2:=\BiCay(G,R,R^{-1},\{1,a\},\{a,b\})$ with $T_{0,1}=\{1,a\}$ and $T_{1,0}=\{a,b\}$. Let $\mathcal{A}:=\Aut(\G_2)$. Then $\G_2$ is an oriented $2$-Cayley digraph of $G$ with valency $|R|+2$.
 In the same way as Lemma~\ref{lem=ORR}, we first show that $\mathcal{A}$ is not vertex-transitive by counting $|A([\G_2^+(1_0)])|$ and $|A([\G_2^+(1_1)])|$. In the beginning, we find the following:
\[\begin{array}{ll}
\G_2^+(1_0)=\{r_0,1_1,a_1~|~r\in R\},~~~&\G_2^-(1_0)=\{r^{-1}_0,b_1,a^{-1}_1~|~r\in R\}, \\
\G_2^+(1_1)=\{r_1,a_0,b_0~|~r\in R\}, ~~~&\G_2^-(1_1)=\{r^{-1}_1,a^{-1}_0,1_0~|~r\in R\}, \\
\G_2^+(a_1)=\{(ra)_1,a^2_0,(ba)_0~|~r\in R\},~~~&\G_2^-(a_1)=\{(r^{-1}a)_1,1_0,a_0~|~r\in R\},\\
\G_2^+(a_0)=\{(ra)_0,a_1,a^2_1~|~r\in R\},~~~&\G_2^-(a_0)=\{(r^{-1}a)_0,(ba)_1,1_1~|~r\in R\},\\
\G_2^+(b_0)=\{(rb)_0,b_1,(ab)_1~|~r\in R\},~~~& \G_2^-(b_0)=\{(r^{-1}b)_0,1_1,(a^{-1}b)_1~|~r\in R\}.\\
\end{array}
\]

Using the above sets, (\ref{Eq18}) and $bh\in H$ (for $h\in H$), it is easy to find the following:
\[
\begin{array}{lll}
\G_2^+(1_1)\cap \G_2^{+}(1_{0})=\{a_0,a_1\},~& \G_2^-(1_1)\cap \G_2^+(1_0) =\emptyset, ~&\G_2^-(a_1)\cap \G_2^{+}(1_{0})=\{1_1,a_0\},~\\
\G_2^+(a_1)\cap \G_2^+(1_0)=\{a^2_0\}~~ \mbox{ if } a^2\in R,~& \G_2^+(a_1)\cap \G_2^+(1_0) = \emptyset~~ \mbox{ if }a^2\notin R ,~& \\
\G_2^+(b_0)\cap \G_2^{+}(1_{1})=\emptyset,~&\G_2^-(b_0)\cap \G_2^+(1_1)=\emptyset,~&\G_2^-(a_0)\cap \G_2^{+}(1_{1})=\emptyset, \\
\G_2^+(a_0)\cap \G_2^+(1_1)=\{a_1,a^2_1\} ~~\mbox{ if } a^2\in R,~&\G_2^+(a_0)\cap \G_2^+(1_1)=   \{a_1\}~~ \mbox{ if }a^2\notin R.~&\\
\end{array}
\]

As $[\G_2^+(1_0)\cap G_0]\cong[\Sigma^+(1)]$, we obtain
\begin{eqnarray*}
|A([\G_2^+(1_0)])| &= &k-|A(1_1,a_{1})|+|\G_2^+(1_1)\cap \G_2^+(1_0)|+|\G_2^-(1_1)\cap \G_2^+(1_0)|+|\G_2^+(a_1)\cap  \G_2^+(1_0)|+|\G_2^-(a_1)\cap \G_2^+(1_0)|.
\end{eqnarray*}
As $(1_1,a_1)$ is an arc in $\G_2$,
$|A(1_1,a_1)|=1$ and thus
\begin{eqnarray}\label{Eq19}
A|[\G_2^+(1_0)]|=k+4 \mbox{ if } a^2\in R \mbox{~~and~~~}A|[\G_2^+(1_0)]|=k+3 \mbox{ if } a^2\notin R
\end{eqnarray}
hold. As $[\G_2^+(1_1)\cap G_1]\cong[\Sigma^+(1)]$, we obtain
\begin{eqnarray*}
|A([\G_2^+(1_1)])| &= &k-|A(a_0,b_{0})|+|\G_2^+(a_0)\cap \G_2^+(1_1)|+|\G_2^-(a_0)\cap \G_2^+(1_1)|+|\G_2^+(b_0)\cap  \G_2^+(1_1)|+|\G_2^-(b_0)\cap \G_2^+(1_1)|.
\end{eqnarray*}
As $|A(a_0,b_0)|=0$, we can find
\begin{eqnarray}\label{Eq20}
A|[\G_2^+(1_1)]|=k+2 \mbox{ if } a^2\in R \mbox{ and }A|[\G_2^+(1_1)]|= k+1\hspace{1em} \mbox{ if } a^2\notin R .
\end{eqnarray}

By (\ref{Eq19})--(\ref{Eq20}),  $[\G_2^+(1_0)]\ncong[\G_2^+(1_1)]$ holds and thus $\G_2$ is not vertex-transitive. Recall that $R(G)$ is a group of automorphisms of $\G_2$ with two orbits $G_0$ and $G_1$. Then $\mathcal{A}$ also has two orbits $G_0$ and $G_1$ satisfying $\mathcal{A}=R(G)\mathcal{A}_{1_1}$. Note that $G_i=H_i \cup (Hb)_i~(i=0,1)$ and $[H_i]\cong [(Hb)_i]\cong \Sigma$. Since $\mathcal{A}_{1_1}$ fixes $G_1$ setwise and $[H_1]$ is an ORR of $H$, $\mathcal{A}_{1_1}$ fixes $H_1$ pointwise. Since $[\G_2^-(1_1)]=((a^{-1})_0,1_0)$ is an arc, $\mathcal{A}_{1_1}$ fixes $1_0$. Since $[H_0]$ is an ORR, $\mathcal{A}_{1_1}$ fixes $H_0$ pointwise. On the other hand, it follows by $[\G_2^+(1_1)]=(a_0,(ab)_0)$ that $\mathcal{A}_{1_1}$ fixes $(ab)_0$ and thus it fixes $(Hb)_0$ pointwise. By $[\G_2^-(b_0)]=((ab)_1,a_1)$, $\mathcal{A}_{1_1}$ fixes $(Hb)_1$ pointwise. This shows $\mathcal{A}_{1_1}=1$ and thus $\mathcal{A}=R(G)\mathcal{A}_{1_1}=R(G)$. Therefore, $\G_2$ is an O$2$SR of $G$. This completes the proof for Subcase 1.1.\\

\f{\bf Subcase 1.2: }For $m\geq 3$, show that $G$ admits an O$m$SR. \\
For $m\geq 3$, define a subset $T_{i,j}\subseteq G~(i,j \in \mz_m)$ as follows:
\[\begin{array}{ll}
T_{0,0}=T_{1,1}=T_{m-1,m-1}=R,~&T_{1,0}=\{a\}, T_{0,1}=\{1\}, T_{m-2,m-1}=\{b\},T_{m-1,0}=\{a^{-1}\}~; \\
T_{i,i} =R^{-1} &\,\,\,\,\,\,\,\,\,\,\,\,\,\,\,\,\,\,\,\,\,\,\,\,\,\,\,\,\,\,\,\,\mbox{ for } i\neq0,1,m-1; \\
T_{i,i-1}=\{1\} &\,\,\,\,\,\,\,\,\,\,\,\,\,\,\,\,\,\,\,\,\,\,\,\,\,\,\,\,\,\,\,\,\mbox{ for } i\neq1; \\
T_{i,i+1}=\{a\} &\,\,\,\,\,\,\,\,\,\,\,\,\,\,\,\,\,\,\,\,\,\,\,\,\,\,\,\,\,\,\,\,\mbox{ for } i\neq0,m-2,m-1; \\
T_{i,j}=\emptyset &\,\,\,\,\,\,\,\,\,\,\,\,\,\,\,\,\,\,\,\,\,\,\,\,\,\,\,\,\,\,\,\,\mbox{ for } i\neq j,j\pm 1.
\end{array}
\]
Let $\G_m:=\Cay(G,T_{i,j}:i,j \in \mz_m)$. Then $\G_m$ is an oriented $m$-Cayley digraph of $G$ with valency $|R|+2$ (See {\sc Figure}~\ref{Fig5}). To complete the proof for Subcase 1.2, we need the following claim that counts the number of arcs in the induced digraph $[\G_m^+(1_i)]$ for $i\in \mz_m$.
\begin{figure}
  \centering
  \includegraphics[width=9cm]{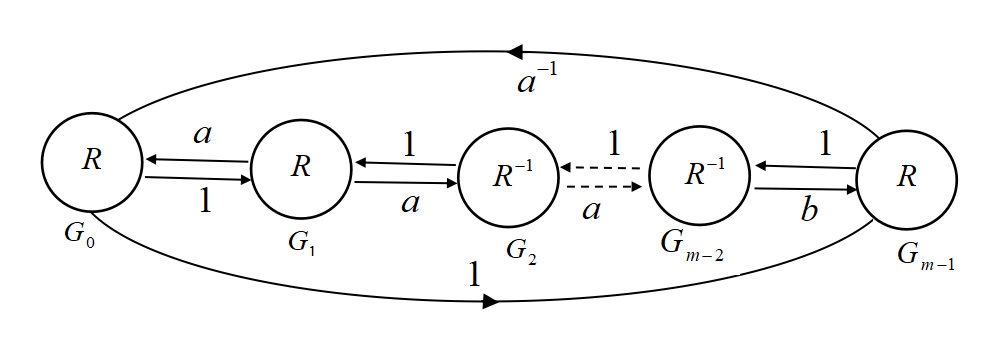}
\caption{The digraph $\G_m$~for~$m\geq 3$}\label{Fig5}
\end{figure}

\smallskip
\f{\bf Claim:}
\smallskip

\f(i) $|A([\G_3^+(1_{0})])|=k+2$, $|A([\G_m^+(1_{0})])|=k+1~(m\geq4).$\\
(ii) $|A([\G_3^+(1_{1})])|=k+1$, $|A([\G_m^+(1_{1})])|=k+2~(m\geq4).$\\
(iii) $|A([\G_m^+(1_{m-1})])|=k~(m\geq 3).$\\
(iv) $|A([\G_m^+(1_{m-2})])|=k~(m\geq 4).$\\
(v) $|A([\G_m^+(1_{i})])|=k~(2\leq i\leq m-3)$.
\medskip

\noindent{\em Proof of Claim}. With the aid of {\sc Figure}~\ref{Fig5}, we find the following:
\[\begin{array}{ll}
\G_m^+(1_0)=\{r_{0},1_1,1_{m-1}~|~r\in R\}, ~~~&\G_m^-(1_0)=\{r^{-1}_{0},a^{-1}_1,a_{m-1}~|~r\in R\},\\
\G_3^+(1_1)=\{r_{1},b_2,a_{0}~|~r\in R\},~~~&\\
\G_m^+(1_1)=\{r_{1},a_2,a_{0}~|~r\in R\}~(m\geq 4),~~~&\G_m^-(1_1)=\{r^{-1}_{1},1_2,1_{0}~|~r\in R\},\\
\G_m^+(1_{m-2})=\{r^{-1}_{m-2},b_{m-1},1_{m-3}~|~r\in R\}~(m\geq 4),~~~&\\
\G_m^+(1_{m-1})=\{r_{m-1},a^{-1}_0,1_{m-2}~|~r\in R\},~~~&\G_m^-(1_{m-1})=\{r^{-1}_{m-1},1_0,b_{m-2}~|~r\in R\},\\
\G_m^+(1_i)=\{r^{-1}_{i},a_{i+1},1_{i-1}~|~r\in R\}~(2\leq i\leq m-3), ~~~&\G_m^-(1_i)=\{r_{i},1_{i+1},a^{-1}_{i-1}~|~r\in R\}~(2\leq i\leq m-2),\\
\G_m^+(a_0)=\{(ra)_0,a_1,a_2~|~r\in R\},~~~&\G_m^-(a_0)=\{(r^{-1}a)_0,1_1,a^2_2~|~r\in R\},\\
\G_m^+(a_{m-2})=\{(r^{-1}a)_{m-2},(ba)_{m-1},a_{m-3}~|~r\in R\}~(m\geq 4),~~~& \\
\G_m^+(a_{i+1})=\{(r^{-1}a)_{i+1},a^2_{i+2},a_{i}~|~r\in R\}~(1\leq i\leq m-4),~~~&\G_m^-(a_{i+1})=\{(ra)_{i+1},a_{i+2},1_{i}~|~r\in R\}~(1\leq i\leq m-3),\\
\G_m^+(a^{-1}_0)=\{(ra^{-1})_0,a^{-1}_1,a^{-1}_{m-1}~|~r\in R\},~~~&\G_m^-(a^{-1}_0)=\{(r^{-1}a^{-1})_0,a^{-2}_1,1_{m-1}~|~r\in R\},\\
\G_m^+(b_{m-1})=\{(rb)_{m-1},(a^{-1}b)_0,b_{m-2}~|~r\in R\},~~~&\G_m^-(b_{m-1})=\{(r^{-1}b)_{m-1},b_{0},1_{m-2}~|~r\in R\}.\\
\end{array}
\]

The above sets enable us to prove the claim as follows.

\noindent{\em Proof of (i). }It is easy to find by the above sets and (\ref{Eq18}) that
\[\begin{array}{llll}
\G_3^+(1_1)\cap \G_3^+(1_0)=\{a_0\}, &\G_3^+(1_{2})\cap \G_3^+(1_0)=\{1_1\},
&\G_3^-(1_1)\cap \G_3^+(1_0)=\{1_2\}, &\G_3^-(1_2)\cap \G_3^+(1_0)=\emptyset~;\\
\G_m^+(1_1)\cap \G_m^+(1_0)=\{a_0\}, &\G_m^+(1_{m-1})\cap \G_m^+(1_0)=\emptyset,
&\G_m^-(1_1)\cap \G_m^+(1_0)=\emptyset, &\G_m^-(1_{m-1})\cap \G_m^+(1_0)=\emptyset~( m\geq 4).
\end{array}\]
As $[\G_m^+(1_0)\cap G_0]\cong[\Sigma^+(1)]$, we obtain
\begin{eqnarray*}
|A([\G_m^+(1_0)])| &= &k-|A(1_1,1_{m-1})|+\\&&|\G_m^+(1_1)\cap \G_m^+(1_0)|+|\G_m^-(1_1)\cap \G_m^+(1_0)|+|\G_m^+(1_{m-1})\cap  \G_m^+(1_0)|+|\G_m^-(1_{m-1})\cap \G_m^+(1_0)|.
\end{eqnarray*}
The result (i) follows by $|A(1_1,1_{2})|=1$  and $|A(1_1,1_{m-1})|=0~(m\geq 4)$  (See {\sc Figure}~\ref{Fig5}).\\

\noindent{\em Proof of (ii). }By (\ref{Eq18}), we have
\[\begin{array}{llll}
\G_3^+(a_0)\cap \G_3^{+}(1_{1})=\{a_1\},~~~&\G_3^+(b_2)\cap \G_3^+(1_1)=\emptyset,~~~&\G_3^-(a_0)\cap \G_3^+(1_1)=\emptyset,~~~&\G_3^-(b_2)\cap \G_3^{+}(1_{1})=\emptyset~;\\
\G_m^+(a_0)\cap \G_m^{+}(1_{1})=\{a_1\},~~~&\G_m^+(a_2)\cap \G_m^+(1_1)=\{a_1\},~~~&
\G_m^-(a_0)\cap \G_m^+(1_1)=\emptyset,~~~&\G_m^-(a_2)\cap \G_m^{+}(1_{1})=\emptyset~(m\geq 4).
\end{array}\]

As $[\G_3^+(1_1)\cap G_1]
\cong [\Sigma^+(1)]$, we have
\begin{eqnarray*}
|A([\G_3^+(1_1)])| &= &k-|A(b_2,a_{0})|+|\G_3^+(b_2)\cap \G_3^+(1_1)|+|\G_3^-(b_2)\cap \G_3^+(1_1)|+|\G_3^+(a_0)\cap  \G_3^+(1_1)|+|\G_3^-(a_0)\cap \G_3^+(1_1)|.
\end{eqnarray*}
This shows $A|[\G_3^+(1_1)]|=k+1$ by  $|A(a_0,b_{2})|=0$ (see {\sc Figure}~\ref{Fig5}).\\
 On the other hand, consider $m\geq4$. As $[\G_m^+(1_1)\cap G_1]
\cong [\Sigma^+(1)]$, we also obtain
\begin{eqnarray*}
|A([\G_m^+(1_1)])| &= &k-|A(a_2,a_{0})|+|\G_m^+(a_2)\cap \G_m^+(1_1)|+|\G_m^-(a_2)\cap \G_m^+(1_1)|+|\G_m^+(a_0)\cap  \G_m^+(1_1)|+|\G_m^-(a_0)\cap \G_m^+(1_1)|.
\end{eqnarray*}
Now Claim~(ii) follows by $|A(a_0,a_{2})|=0~(m\geq 4)$ (see {\sc Figure}~\ref{Fig5}).\\

\noindent{\em Proof of (iii). }The result (iii) follows by $|A(a^{-1}_0,1_{m-2})|=0~(m\geq 3)$ and
\begin{eqnarray*}
\G_m^+(a^{-1}_0)\cap \G_m^{+}(1_{m-1})=\emptyset,~~~\G_m^+(1_{m-2})\cap \G_m^{+}(1_{m-1})=\emptyset,
~~~\G_m^-(a_0)\cap \G_m^+(1_{m-1})=\emptyset,~~~\G_m^-(1_{m-2})\cap \G_m^+(1_{m-1})=\emptyset. \\
\end{eqnarray*}

\noindent{\em Proof of (iv). }The result (iv) follows by $|A(b_{m-1},1_{m-3})|=0~(m\geq 4)$ and
\begin{eqnarray*}
\G_m^+(b_{m-1})\cap \G_m^{+}(1_{m-2})=\emptyset,~~~\G_m^+(1_{m-3})\cap \G_m^{+}(1_{m-1})=\emptyset,~~~
\G_m^-(b_{m-1})\cap \G_m^+(1_{m-2})=\emptyset,~~~\G_m^-(1_{m-3})\cap \G_m^+(1_{m-1})=\emptyset.\\
\end{eqnarray*}

\noindent{\em Proof of (v). }Let $m\geq 5$. For each $2\leq i\leq m-3$, we find
\begin{eqnarray*}
\G_m^+(a_{i+1})\cap \G_m^{+}(1_{i})=\emptyset,~~~\G_m^+(1_{i-1})\cap \G_m^{+}(1_{i})=\emptyset,~~~
\G_m^-(a_{i+1})\cap \G_m^+(1_{i})=\emptyset,~~~\G_m^-(1_{i-1})\cap \G_m^+(1_{i})=\emptyset.
\end{eqnarray*}
The result (v) follows by $|A(a_{i+1},1_{i-1})|=0$ ({\sc Figure}~\ref{Fig5}). This completes the proof of the claim.\qed\\

Now we are ready to complete the proof of Subcase 1.2 in Lemma \ref{lem=generalizeddihedral} by using the above claim. Let $\mathcal{A}:=\Aut(\G_m)$.\\
By $|A([\G_m^+(1_0)])|\ne |A([\G_m^+(1_i)])|~(i\neq 0)$ and $|A([\G_m^+(1_1)])|\ne |A([\G_m^+(1_j)])|~(j\neq 1)$ (see the above claim), we obtain $$[\G_m^+(1_0)]\ncong [\G_m^+(1_i)]~(i\neq 0) \mbox{~and~}[\G_m^+(1_1)]\ncong [\G_m^+(1_j)]~(j\neq 1).$$ Thus $\mathcal{A}$ fixes $G_0$ and $G_1$ setwise. By $|T_{1,2}|=1$ and $T_{1,i}=\emptyset~(3\leq i\leq m-1)$, $\mathcal{A}$ fixes $G_2$ setwise. Similarly, we obtain that $\mathcal{A}$ fixes $G_i$ setwise for each $i\in\mz_m$. Recall that $R(G)$ is a group of automorphisms of $\G_m$ with $m$ orbits  $G_i~(i\in\mz_m)$. Then $\mathcal{A}$ also has $m$ orbits satisfying $\mathcal{A}=R(G)\mathcal{A}_{1_0}$. To prove $\G_m$ is an O$m$SR, we only need to show $\mathcal{A}_{1_0}=1$. Let $\sigma \in \mathcal{A}_{1_0}$. Then $\sigma$ fixes $G_i$ setwise, where $G_i=H_i\cup (Hb)_i~(i\in \mz_m)$. Since $[H_0]\cong \Sigma$ is an ORR of $H$, $\sigma$ fixes $H_0$ pointwise. By $T_{0,1}=\{1\}$, $\mathcal{A}_{1_0}$ fixes $H_{1}$ pointwise. By $T_{j,j+1}=\{a\}~(1\leq j\leq m-3)$, $\mathcal{A}_{1_0}$ fixes $H_i~(2\leq i\leq m-2)$ pointwise. As $T_{m-2,m-1}=\{b\}$, $\mathcal{A}_{1_0}$ fixes $H_{{m-1}}$ pointwise. On the other hand, it follows by $T_{m-2,m-1}=\{b\}$ that $\mathcal{A}_{1_0}$ fixes $(Hb)_{m-1}$ pointwise. By $T_{m-1,0}=\{a^{-1}\}$, $\mathcal{A}_{1_0}$ fixes $(Hb)_{0}$ pointwise. We can also obtain by $T_{j,j+1}=\{a\}~(1\leq j\leq m-3)$ that $\mathcal{A}_{1_0}$ fixes $(Hb)_i~(1\leq i\leq m-2)$ pointwise. As $G=H\cup Hb$, we conclude that $\mathcal{A}_{1_0}$ fixes $G_i~(i\in \mz_m)$ pointwise (i.e., $\mathcal{A}_{1_0}=1$). Therefore Subcase 1.2 follows. This completes the proof of Case 1.\qed \\

\f{\bf Case 2:} $H$ does not admit ORRs.\\
\noindent{\em Proof of Case 2. }Recall that $H$ is an abelian group and $G$ is a non-abelian generalized dihedral group over $H$. By Remark~\ref{rmkgdg} and Proposition~\ref{prop=ORR}, $H=\mz_4 \times \mz_2$, $\mz_3^2$,  $\mz_4 \times \mz_2^2$, $\mz_3 \times \mz_2^3$, $\mz_4 \times \mz_2^3$ or $\mz_4 \times \mz_2^4$. For each group $H$, take two subsets $R,L \subseteq H$ as follows:
\begin{enumerate}
\item $R=\{x,xb\}$ and $L=\{x,x^{-1}y\}$ for $H=\mz_4\times \mz_2=\langle x,y|~x^4=y^2=1,xy=yx\rangle$~;
\item $R=\{x,xb\}$ and $L=\{x,x^{-1}y\}$ for $H=\mz_3^2=\langle x,y|~x^3=y^3=1,xy=yx\rangle$ ~;
\item $R=\{x,xy\}$ and $L=\{x,x^{-1}yz\}$ for $H=\mz_4\times \mz_2^2=\langle x,y,z~|~x^4=y^2=z^2=1,xy=yx,xz=zx,yz=zy\rangle$~;
\item $R=\{x,xy,xz,xw\}$ and $L=\{x,xz,xzw,x^{-1}yzw\}$ \\
for $H=\mz_3\times \mz_2^3=\langle x,y,z,w~|~x^3=y^2=z^2=w^2=1,xy=yx,xz=zx,yz=zy,xw=wx,yw=wy,zw=wz\rangle$~;
\item $R=\{x,xy,xz,xw\}$ and $L=\{x,xz,xzw,x^{-1}yzw\}$ \\
for $H=\mz_4\times \mz_2^3=\langle x,y,z,w~|~x^4=y^2=z^2=w^2=1,xy=yx,xz=zx,yz=zy,xw=wx,yw=wy,zw=wz\rangle$~;
\item $R=\{x,xy,xz,xw,xu\}$ and $L=\{x,xz,x^{-1}zw,xyzw,x^{-1}zwu\}$ \\
for $H=\mz_4\times \mz_2^4=\langle x,y,z,w,u~|~x^4=y^2=z^2=w^2=u^2=1,xy=yx,xz=zx,yz=zy,xw=wx,yw=wy,zw=wz,xu=ux,yu=uy,zu=uz,wu=uw\rangle$.
\end{enumerate}
Note that for each case, $R$ and $L$ satisfy
$$|R|=|L| \mbox{~and~}  R\cap R^{-1}=L\cap L^{-1}=\emptyset.$$
For each case in (1)-(6) and for each $m\geq 2$, we define a subset $T_{i,j}\subseteq G~(i,j\in \mz_m)$ as follows:
\[
\begin{array}{lll}
T_{0,0}=R,~~~&T_{1,0}=\{x^{-1}\},~~~&T_{0,1}=\{b\}~;\\
T_{i,i}=L ,~~~&T_{i,i+1}=\{1\}~~~&\mbox{for } i\neq 0~; \\
&T_{i,i-1}=\{x\}~~~&\mbox{for } i\neq1~; \\
&T_{i,j}=\emptyset ~~~& \mbox{for } i\neq j,j\pm1.
\end{array}
\]
Let $\G_m:=\Cay(G,T_{i,j}:i,j \in \mz_m)$ and $\mathcal{A}:=\Aut(\G_m)$. Then $\G_m$ is an oriented $m$-Cayley digraph of $G$, where the valency of $\G_m$ is $|R|+1$ for $m=2$ and $|R|+2$ for $m\geq 3$ (See {\sc Figure}~\ref{Fig4}).
\begin{figure}
 \centering
  \includegraphics[width=9cm]{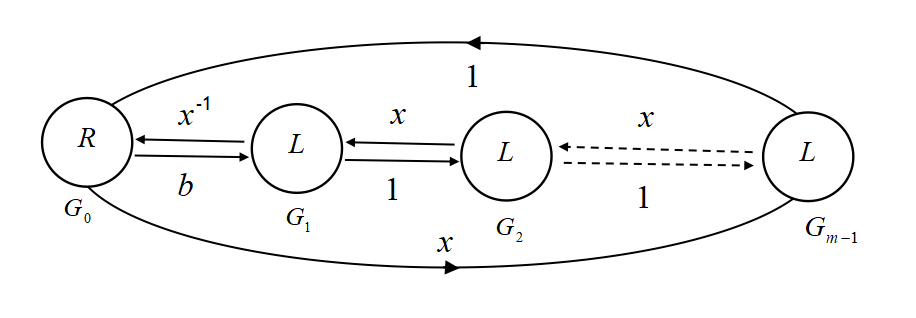}
\caption{The digraph $\G_m$~for $m\geq 2$}\label{Fig4}
\end{figure}
With the aid of {\sc Magma} \cite{magma}, we can find that $\G_m$ is an O$m$SR of $G$ for $m=2,3$. In the rest of the proof, we will show that $\G_m$ is an O$m$SR of $G$ for all $m\geq 4$. Let $\Sigma:=\Cay(G,R)$ and $\Phi:=\Cay(G,L)$. It is easy to see that $|A([\Sigma^+(1)])|=|A([\Phi^+(1)])|=0$ holds for each case.
By the definition of $\G_m$ and $R(g)\in \mathcal{A}$ for $g\in G$, we find the following:
\[
\begin{array}{ll}
\G_m^+(1_0)=\{r_0,b_1,x_{m-1}~|~r\in R\},~~~& \G_m^-(1_0)=\{r^{-1}_0,x_1,1_{m-1}~|~r\in R\},\\
\G_m^+(1_1)=\{l_1,1_2,x^{-1}_0~|~l\in L\},~~~&\G_m^-(1_1)=\{l^{-1}_1,x^{-1}_2,b_0~|~l\in L\},\\
\G_m^+(1_i)=\{l_i,1_{i+1},x_{i-1}~|~l\in L\}~(2\leq i\leq m-1),~~~&\G_m^-(1_i)=\{l^{-1}_i,x^{-1}_{i+1},1_{i-1}~|~l\in L\}~(2\leq i\leq m-1),\\
\G_m^+(b_1)=\{(lb)_1,b_2,(x^{-1}b)_0~|~l\in L\},~~~&\G_m^-(b_1)=
\{(l^{-1}b)_1,(x^{-1}b)_2,1_0~|~l\in L\},\\
\G_m^+(x_{1})=\{(lx)_{1},x_2,1_{0}~|~l\in L\},~~~&  \G_m^-(x_{1})=\{(l^{-1}x)_{1},1_2,(bx)_0~|~l\in L\},\\
\G_m^+(x_{i-1})=\{(lx)_{i-1},x_{i},x^2_{i-2}~|~l\in L\}~(3\leq i\leq m-1),~~~& \G_m^-(x_{i-1})=\{(l^{-1}x)_{i-1},1_{i},x_{i-2}~|~l\in L\}~(3\leq i\leq m-1),\\
\G_m^+(x_{m-1})=\{(lx)_{m-1},x_0,x^2_{m-2}~|~l\in L\},~~~&\G_m^-(x_{m-1})=\{(l^{-1}x)_{m-1},1_0,x_{m-2}~|~l\in L\},\\
\G_m^+(x^{-1}_0)=\{(rx^{-1})_{0},(bx^{-1})_1,1_{m-1}~|~r\in R\},~~~& \G_m^-(x^{-1}_0)=\{(r^{-1}x^{-1})_{0},1_1,x^{-1}_{m-1}~|~r\in R\}.\\
\end{array}
\]
Using the above sets and  $x\in R\cap L$, $b\notin R\cup L$, $1\notin R$ and $1\notin L$, it is easy to find the following:
\[
\begin{array}{ll}
\G_m^+(b_1)\cap \G_m^+(1_0)=\emptyset, ~~~& \G_m^-(b_1)\cap \G_m^+(1_0)=\emptyset, \\
\G_m^+(x_{m-1})\cap \G_m^{+}(1_0)=\{x_0\},~~~&\G_m^-(x_{m-1})\cap \G_m^{+}(1_0)=\emptyset,\\
\G_m^+(1_2)\cap \G_m^+(1_1)=\{x_1\}, ~~~& \G_m^-(1_2)\cap \G_m^+(1_1)=\emptyset, \\
\G_m^+(x^{-1}_{0})\cap \G_m^{+}(1_1)=\emptyset, ~~~&\G_m^-(x^{-1}_{0})\cap \G_m^{+}(1_1)=\emptyset,\\
\G_m^+(1_{3})\cap \G_m^+(1_{2})=\{x_2\}, ~~~& \G_m^-(1_{3})\cap \G_m^+(1_{2})=\emptyset,\\
\G_m^+(x_{1})\cap \G_m^+(1_{2})=\{x_2\}, ~~~&\G_m^-(x_{1})\cap \G_m^+(1_{2})=\emptyset,\\
\G_m^+(1_{i+1})\cap \G_m^+(1_{i})=\{x_i\}~~(3\leq i\leq m-1),~~~& \G_m^-(1_{i+1})\cap \G_m^+(1_{i})=\emptyset~~(3\leq i\leq m-1), \\
\G_m^+(x_{i-1})\cap \G_m^+(1_{i})=\{x_i\}~~(3\leq i\leq m-1),~~~&\G_m^-(x_{i-1})\cap \G_m^+(1_{i})=\emptyset~~(3\leq i\leq m-1).
\end{array}
\]
In the same way as Lemma~\ref{lem=ORR}, we first show
\begin{equation}\label{Eq151617}
|A([\G_m^+(1_i)])|=\left\{
\begin{array}{ll}
1 & \mbox{~~~for~~} i=0,1\\
2 & \mbox{~~~for~~} i\ne 0,1
\end{array}
\right..
\end{equation}
Let $i=0$. As $|A(R_0)|=|[\G_m^+(1_0)\cap G_0]|=|A([\Sigma^+(1)])|=0$ and $|A(b_1,x_{m-1})|=0$, we find (\ref{Eq151617}) for $i=0$ by
\[|A([\G_m^+(1_0)])|= 0-|A(b_1,x_{m-1})|+
 |\G_m^+(b_1)\cap \G_m^+(1_0)|+|\G_m^-(b_1)\cap \G_m^+(1_0)|+|\G_m^+(x_{m-1})\cap  \G_m^+(1_0)|+|\G_m^-(x_{m-1})\cap \G_m^+(1_0)|.
\]
Let $i=1$. As $|A(L_1)|=|[\G_m^+(1_1)\cap G_1]|=|A( [\Phi^+(1)])|=0$ and $|A(1_2,x^{-1}_0)|=0$, we find (\ref{Eq151617}) for $i=1$ by
\[|A([\G_m^+(1_1)])|= 0-|A(1_{2},x^{-1}_{0})|+
 |\G_m^+(1_{2})\cap \G_m^+(1_1)|+|\G_m^-(1_{2})\cap \G_m^+(1_1)|+|\G_m^+(x^{-1}_{0})\cap  \G_m^+(1_1)|+|\G_m^-(x^{-1}_{0})\cap \G_m^+(1_1)|.
\]
Finally, we consider $i\neq 0,1$. Since there are no arcs between $1_{i+1}$ and $x_{i-1}$, we have $|A(1_{i+1},x_{i-1})|=0~(i\ne 0,1)$. Now (\ref{Eq151617}) for $i\ne 0,1$ follows immediately.\\
Since $[\G_m^+(1_0)]\ncong [\G_m^+(1_i)]$ and $[\G_m^+(1_1)]\ncong [\G_m^+(1_i)]~(i\ne 0,1)$ all hold by (\ref{Eq151617}), $\mathcal{A}$ fixes $G_0\cup G_1$ setwise.
Since $[G_0\cup G_1]\cong\G_2$ is an O$2$SR, $\mathcal{A}$ fixes $G_i~(i=0,1)$ setwise. By $T_{1,2}=\{1\}$ and $T_{1,i}=\emptyset~(i\neq 0,1,2)$, $\mathcal{A}$ fixes $G_2$ setwise. Similarly, $\mathcal{A}$ fixes $G_i$ setwise for each $i\in \mz_m$. Since $[G_0\cup G_1]\cong\G_2$ is an O$2$SR, $\mathcal{A}_{1_0}$ fixes $G_i~(i\in \mz_m)$ pointwise and so $\mathcal{A}_{1_0}=1$. Hence $\mathcal{A}=R(G)\mathcal{A}_{1_0}=R(G)$ and so $\G_m$ is an O$m$SR of $G$. This completes the proof for Case 2 of Lemma \ref{lem=generalizeddihedral}.
\end{proof}

Now, we are ready to prove Theorem~\ref{theo=main}.

\begin{proof}[\bf Proof of Theorem~\ref{theo=main}]
Let $G$ be a finite group and $m\geq 1$ be an integer. Without loss of generality, we may assume that $G$ does not satisfy Theorem~\ref{theo=main}(1), that is, $G$ does not admit O$m$SRs. If $m=1$ then Theorem~\ref{theo=main}(2) follows by Proposition~\ref{prop=ORR}. Now assume $m\geq 2$ and $G$ does not admit O$m$SRs. To complete the proof, we consider two cases, $G$ admits an ORR or not.\\

\indent If $G$ admits an ORR, then $G$ is $\mz_1$ or $\mz_2$ by Lemma~\ref{lem=ORR}. Hence, we obtain that $G$ is either $\mz_1$ with $3\leq m\leq 6$ (i.e., Theorem~\ref{theo=main}(4)) or $\mz_1, \mz_2$ with $m=2$ by Lemma \ref{lem=elementary abelian 2-group order le 16}.\\

 \indent If $G$ does not admit ORRs, then it follows by Proposition~\ref{prop=ORR} that $G$ is either a generalized dihedral group of order greater than $2$ or one of the $11$ exceptional groups given in {\sc Table}~\ref{table1}. We first consider that $G$ is a generalized dihedral group of order greater than $2$. Then $G$ is an elementary abelian $2$-group by Remark~\ref{rmkgdg} and Lemma \ref{lem=generalizeddihedral}. Moreover, $G$ has order at most $2^4=16$ by Lemma \ref{lem=elementary abelian 2-group}. Now, Theorem~\ref{theo=main}(3) follows by Lemma~\ref{lem=elementary abelian 2-group order le 16}. To complete the proof, we now show that none of the $11$ exceptional groups in {\sc Table}~\ref{table1} does not admit O$m$SRs for any $m\geq 2$. For each group $G$  in {\sc Table}~\ref{table1}, take two subsets $R,L\subseteq G$ as follows:
\begin{enumerate}
\item $R=\{x,xy\}$ and $L=\{x,x^{-1}y\}$ for $G=\mz_4\times \mz_2=\langle x,y~|~x^4=y^2=1,xy=yx\rangle$~;
\item $R=\{x,xy\}$ and $L=\{x,x^{-1}y\}$ for $G=Q_8=\langle x,y~|~x^4=y^4=1,x^2=y^2,x^y=x^{-1}\rangle$~;
\item $R=\{x,xy\}$ and $L=\{x,x^{-1}y\}$ for $G=\mz_3^2=\langle x,y~|~x^3=y^3=1,xy=yx\rangle$~;
\item $R=\{x,xy,xz\}$ and $L=\{x,x^{-1}y,x^{-1}yz\}$ \\
for $G=\mz_4\times \mz_2^2=\langle x,y,z~|~x^4=y^2=z^2=1,xy=yx,xz=zx$, $yz=zy\rangle$~;
\item $R=\{x,xy,xz,xw\}$ and $L=\{x,x^{-1}y,x^{-1}yz,x^{-1}yzw\}$ \\
for  $G=\mz_3\times \mz_2^3=\langle x,y,z,w~|~x^3=y^2=z^2=$ $w^2=1,xy=yx,xz=zx,yz=zy,xw=wx,yw=wy,zw=wz\rangle$~;
\item $R=\{x,xy,xz,xw\}$ and $L=\{x,x^{-1}y,x^{-1}yz,x^{-1}yzw\}$ \\
for $G=\mz_4\times \mz_2^3=\langle x,y,z,w~|~x^4=y^2=z^2=$ $w^2=1,xy=yx,xz=zx,yz=zy,xw=wx,yw=wy,zw=wz\rangle$~;
\item $R=\{x,xy,xz,xw,xu\}$ and $L=\{x,x^{-1}y,x^{-1}yz,x^{-1}yzw,x^{-1}yzwu\}$ \\
for $G=\mz_4\times \mz_2^4=\langle x,y,z,w,u~|~x^4=y^2=z^2=w^2=$ $u^2=1,xy=yx,xz=zx,yz=zy,xw=wx,yw=wy,zw=wz,xu=ux,yu=uy,zu=uz,wu=uw\rangle$~;
\item $R=\{x,xy\}$ and $L=\{x,x^{-1}y\}$ for $G=H_1$~;
\item $R=\{x,xy,xz\}$ and $L=\{x,x^{-1}y,x^{-1}yz\}$ for  $G=H_2$~;
\item $R=\{x,xy,xz\}$ and $L=\{x,x^{-1}y,x^{-1}yz\}$ for $G=H_3$~;
\item  $R=\{x,xy,xz,xw\}$ and $L=\{x,x^{-1}y,x^{-1}yz,x^{-1}yzw\}$ \\
for $G=D_4\circ D_4=\langle x,y~|~x^4=y^2=1,x^y=x^{-1}\rangle\circ\langle z,w~|~z^4=w^2,z^w=z^{-1}\rangle$.
\end{enumerate}
Note that for each case, $R$ and $L$ satisfy
\begin{equation*}
|R|=|L| \mbox{~and~} R\cap R^{-1}=L\cap L^{-1}=\emptyset.
\end{equation*}
For a given $G$ in (1)-(11) and for each $m\geq 2$, we define a subset $T_{i,j}\subseteq G~(i,j\in \mz_m)$ as follows:
\[
\begin{array}{ll}
T_{0,0}=R,\,\,\,\,\,\,\,\,\,\,\,\,T_{1,0}=\{x^{-1}\}~; &\nonumber \\
T_{i,i}=L,\,\,\,\,\,\,\,\,\,\,\,\,\,T_{i,i-1}=\{x\}  &\mbox{ for~} i\neq1~; \nonumber\\
T_{i,i+1}= \{1\} &\mbox{ for~} i\in \mz_m~; \nonumber \\
T_{i,j}=\emptyset  &\mbox{ for~} i\ne j,j\pm1. \nonumber
\end{array}
\]
Let $\G_m:=\Cay(G,T_{i,j}:i,j \in \mz_m)$ and $\mathcal{A}:=\Aut(\G_m)$. Then $\G_m$ is an oriented $m$-Cayley digraph of $G$, where the valency of $\G_m$ is
$|R|+1$ for $m=2$ and $|R|+2$ for $m\geq 3$ (See {\sc Figure}~\ref{Fig6}).

\begin{figure}
  \centering
  \includegraphics[width=9cm]{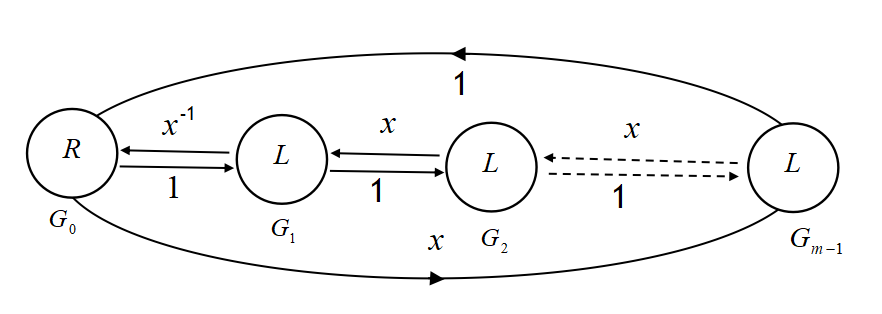}
\caption{The digraph $\G_m$~for $m\geq 2$}\label{Fig6}
\end{figure}
With the aid of {\sc Magma} \cite{magma}, we can find that $\G_m$ is an O$m$SR of $G$ for $m=2,3$. In the rest of the proof, we will show that $\G_m$ is an O$m$SR of $G$ for all $m\geq 4$. Let $\Sigma:=\Cay(G,R)$ and $\Phi:=\Cay(G,L)$. Both $\Sigma$ and $\Phi$ are oriented Cayley digraphs of $G$ satisfying $|A([\Sigma^+(1)])|=|A([\Phi^+(1)])|=0$. By the definition of $\G_m$ and $R(g)\in \mathcal{A}$ for $g\in G$, we find the following:
\[\begin{array}{ll}
\G_m^+(1_0)=\{r_0,1_1,x_{m-1}~|~r\in R\}, &~~\G_m^-(1_0) =\{r^{-1}_0,x_1,1_{m-1}~|~r\in R\},\\
\G_m^+(1_1)=\{l_1,1_2,x^{-1}_{0}~|~l\in L\}, &~~\\
\G_m^+(1_i)=\{l_i,1_{i+1},x_{i-1}~|~l\in L\}~~ (2\leq i\leq m-1), &~~\G_m^-(1_i)=\{l^{-1}_i,x^{-1}_{i+1},1_{i-1}~|~l\in L\}~~ (1\leq i\leq m-1), \\
\G_m^+(x_1)=\{(lx)_1,x_2,1_0~|~l\in L\}, &~~\G_m^-(x_1)=\{(l^{-1}x)_1,1_2,x_0~|~l\in L\},\\
\G_m^+(x_{i-1})=\{(lx)_{i-1},x_{i},x^2_{i-2}~|~l\in L\}~~(3\leq i\leq m-1), &~~\G_m^-(x_{i-1})=\{(l^{-1}x)_{i-1},1_{i},x_{i-2}~|~l\in L\}~~(3\leq i\leq m-1),\\
\G_m^+(x^{-1}_{0})=\{(rx^{-1})_{0},x^{-1}_1,1_{m-1}~|~r\in R\}, &~~\G_m^-(x^{-1}_{0})=\{(r^{-1}x^{-1})_{0},1_1,x^{-1}_{m-1}~|~r\in R\}.
\end{array}\]
Using $x\in R\cap L$, $x^{-1}\notin R\cup L$, $1\notin R\cup L\cup R^{-1}\cup L^{-1}$ and the above sets, it is easy to find the following:
\[\begin{array}{ll}
\G_m^+(1_1)\cap \G_m^+(1_0)=\emptyset,&~~\G_m^-(1_1)\cap \G_m^+(1_0)=\emptyset,\\
\G_m^+(x_{m-1})\cap \G_m^+(1_0)=\{x_0\},&~~\G_m^-(x_{m-1})\cap \G_m^+(1_0)=\emptyset,\\
\G_m^+(x^{-1}_{0})\cap \G_m^+(1_1)=\emptyset,&~~\G_m^-(x^{-1}_{0})\cap \G_m^+(1_1)=\emptyset, \\
\G_m^+(1_{i+1})\cap \G_m^+(1_{i})=\{x_i\}~~(i \ne 0),&~~\G_m^-(1_{i+1})\cap \G_m^+(1_{i})=\emptyset ~~(i \ne 0),\\
\G_m^+(x_{i-1})\cap \G_m^+(1_{i})=\{x_i\}~~( i\not=0,1),&~~\G_m^-(x_{i-1})\cap \G_m^+(1_{i})=\emptyset  ~~( i\ne 0,1).
\end{array}\]

Note that $|A(1_1,x_{m-1})|=|A(1_2,x^{-1}_{0})|=|A(1_{i+1},x_{i-1})|=0$ holds for $2\leq i\leq m-1$. In the same way of claims in Lemma~\ref{lem=ORR} and Lemma~\ref{lem=generalizeddihedral}, we can obtain
$$
|A([\G_m^+(1_{0})])|=1,~~ |A([\G_m^+(1_{1})])|=1,~~|A([\G_m^+(1_{i})])|=2~~(i\ne 0,1).
$$
Since $\mathcal{A}$ fixes $G_0\cup G_1$ setwise and $[G_0\cup G_1]\cong\G_2$ is an O$2$SR, $\mathcal{A}$ fixes $G_i~(i=0,1)$ setwise.
By $T_{1,2}=\{1\}$ and $T_{1,i}=\emptyset$ $(i\neq 0,1,2)$, $\mathcal{A}$ fixes $G_2$ setwise. Similarly, $\mathcal{A}$ fixes $G_i$ setwise for each $i\in \mz_m$. Since $\G_2$ is an O$2$SR, $\mathcal{A}_{1_0}$ fixes $G_i~(i\in \mz_m)$ pointwise and so $\mathcal{A}_{1_0}=1$. Thus, $\mathcal{A}=R(G)\mathcal{A}_{1_0}=R(G)$ and hence $\G_m$ is an O$m$SR of $G$. Therefore, none of the $11$ exceptional groups in {\sc Table}~\ref{table1} does not admit O$m$SRs for any $m\geq 2$. This completes the proof of Themrem~\ref{theo=main}.
\end{proof}

\f {\bf Acknowledgement:} The first author was supported by the National Natural Science Foundation of China (12101601, 12161141005), the Natural Science Foundation of Jiangsu Province, China (BK20200627) and by the China Postdoctoral Science Foundation (2021M693423). The second author was supported by the National Natural Science Foundation of China (12161141005). The third author was supported by Basic Science Research Program through the National Research Foundation of Korea(NRF) funded by the Ministry of Education (NRF-2017R1D1A1B06029987).



\begin{thebibliography}{99}
\bibitem{Babai}
L. Babai, Finite digraphs with given regular automorphism groups,
Period. Math. Hungar. 11 (1980), 257--270.

\bibitem{BabaiI}
L. Babai, W. Imrich, Tournaments with given regular group,
Aequationes Math. 19 (1979), 232--244.


\bibitem{magma}
W. Bosma, C. Cannon, C. Playoust, The MAGMA algebra system I: The user language, J. Symbolic Comput. 24 (1997), 235--265.
\bibitem{DFS}
J.~-L.~Du, Y.~-Q.~Feng, P.~Spiga, A classification of the graphical $m$-semiregular representations of finite groups, J. Combin. Theory Ser. A 171 (2020), 105174.
\bibitem{DFS2}
J.~-L.~Du, Y.~-Q.~Feng, P.~Spiga, On Haar digraphical representations of groups, Discrete Math. 343 (2020), 112032.
\bibitem{DFS3}
J.~-L.~Du, Y.~-Q.~Feng, P.~Spiga,
On $n$-partite digraphical representations of finite groups, J. Combin. Theory Ser. A 189 (2022), 105606.
\bibitem{Godsil}
C.D. Godsil, GRR's for non-solvable groups, in Algebraic Methods in Graph theory (Proc. Conf. Szeged 1978 L. Lov$\acute{a}$sz and V. T. Sos, eds), Coll. Math. Soc. J. Bolyai 25,
North-Holland, Amsterdam, 1981, pp.221--239.
\bibitem{Hujdurovic}
A. Hujdurovi$\acute{c}$, K. Kutnar, D. Maru$\check{s}$i$\check{c}$, On normality of $n$-Cayley graphs, Appl. Math. Comput. 332 (2018), 469--476.

\bibitem{Imrich1}
W. Imrich, Graphs with transitive abelian automorphism group, Coll. Soc. Janos Bolyai 4 (1969), 651--656.

\bibitem{Imrich}
W. Imrich, Graphical regular representations of groups odd order, in: Combinatorics, Coll. Math. Soc. J\'anos. Bolayi 18 (1976), 611--621.

\bibitem{ImrichWatkins}
W. Imrich,  M.E. Watkins, On graphical regular representations of cyclic extensions of groups, Pac. J. Math. 55 (1974), 461--477.

\bibitem{Kovacs}
I. Kov$\acute{a}$cs, A. Malni$\check{c}$, D. Maru$\check{s}$i$\check{c}$, $\check{S}$. Miklavi$\check{c}$, One-matching bi-Cayley graphs over abelian groups, European J. Combin. 30 (2009), 602--616.

\bibitem{KMMS}
K.~Kutnar, D.~Maru\v{s}i\v{c}, S.~Miklavi\v{c}, P.~\v{S}parl,
Strongly regular tri-Cayley graphs,
European J. Combin. 30 (2009), 822--832.
\bibitem{Leemann}
P.-H. Leemann, M. Salle, Cayley graphs with few automorphisms,
J. Algebraic Combin. 53 (2021), 1117--1146.

\bibitem{MorrisSpiga1}
J. Morris, P. Spiga, Every finite non-solvable group admits an oriented regular representation, J. Combin. Theory Ser. B 126 (2017), 198--234.

\bibitem{MorrisSpiga3}J. Morris, P.~Spiga, Classification of finite groups that admit an oriented regular
representation, Bull. Lond. Math. Soc. 50 (2018), 811-831.


\bibitem{MorrisSpiga4}J. Morris, P.~Spiga, Asymptotic enumeration of Cayley digraphs, Israel J. Math. 242 (2021), 401-459.

\bibitem{NowitzWatkins1}
L.~A.~Nowitz, M.~E.~Watkins, Graphical regular representations of non-abelain groups, $I$,
Canad. J. Math. 24 (1972), 994--1008.

\bibitem{NowitzWatkins2}
L.~A.~Nowitz, M.~E.~Watkins, Graphical regular representations of non-abelain groups, $II$,
Canad. J. Math. 24 (1972), 1009--1018.

\bibitem{Spiga}
P. Spiga, Finite groups admitting an oriented regular representation,
J. Combin. Theory Ser. A 153 (2018), 76--97.

\bibitem{Spiga2}
P. Spiga, Cubic graphical regular representations of finite non-abelian simple groups, Commu. Algebra
46 (2018), 2440--2450.

\bibitem{VerretXia}
G. Verret, B.~Xia, Oriented regular representations of out-valency
two for finite simple groups, Ars Math. Contemp. 22 (2022), $\sharp$P1.07.

\bibitem{Xia}
B.~Xia, On cubic graphical regular representations of finite simple groups, J. Combin. Theory Ser. B 141 (2020), 1--30.
\end{thebibliography}
\end{document}